\documentclass[english]{article}

\usepackage[latin1]{inputenc}
\usepackage{tipa}
\usepackage{dsfont}
 
\usepackage{babel}
\usepackage{leftidx}

\usepackage{xcolor}

\usepackage{xr}
 
\usepackage{amsxtra}
\usepackage{amsmath}
\usepackage{amssymb}
\usepackage{amsfonts}
\usepackage[all,2cell,cmtip]{xy}
\UseAllTwocells
\usepackage{mathrsfs}
\usepackage{amsthm}
\usepackage{enumitem}
\usepackage{hyperref}
\usepackage{subfigure}
\usepackage{todonotes}
\usepackage{graphicx}

\xyoption{rotate}

\usepackage{euscript}

\setlist[enumerate]{leftmargin=*,labelindent=.5pc}
\usepackage{fullpage}

\newtheorem{thm}{Theorem}
\newtheorem{cor}[thm]{Corollary}

\newtheorem{lem}[thm]{Lemma}
\newtheorem{prop}[thm]{Proposition}
\newtheorem*{prethm}{Main theorem}
\newtheorem*{precor}{Corollary of main theorem}

\newtheoremstyle{example}{\topsep}{\topsep}%
     {}
     {}
     {\bfseries}
     {.}
     {2pt}
     {\thmname{#1}\thmnumber{ #2}\thmnote{ #3}}

   \theoremstyle{example}
\newtheorem{defi}[equation]{Definition}
\newtheorem{rem}[equation]{Remark}

\newtheorem{ex}[equation]{Example}

\def\on{\operatorname}

\def\Fun{\operatorname{Fun}}

\newcommand{\Wedge}{\mbox{\Large $\wedge$}}

\def\colim{\operatorname*{colim}}

\def\bAut{\operatorname{\bf Aut}}

\def\Aff{\on{Aff}}
\def\Aft{\on{{Aff}_{ft}}}

\def\O{{\mathcal{O}}}

\def\NN{ {\mathbf N}}

\def\Coh{ \on{Coh}}

\def\C{ {\EuScript C}}

\def\Tr{\on{Tr}}

\setlist[enumerate,1]{label=(\arabic{*})}
\setlist[enumerate,2]{label=(\alph{*})}
\setlist[enumerate,3]{label=(\roman{*})}

\def\centerarc[#1](#2)(#3:#4:#5)  { \draw[#1] ($(#2)+({#5*cos(#3)},{#5*sin(#3)})$) arc (#3:#4:#5); }
\tikzstyle{dot}=[draw,circle,fill=black,inner sep=0,minimum size=4pt]
\tikzset{
    partial ellipse/.style args={#1:#2:#3}{
        insert path={+ (#1:#3) arc (#1:#2:#3)}
    }
}

\def\ZZ{\mathbf{Z}}

\def\F{{ F}}

\def\mod#1{\on{mod}}

\def\Perf{\on{Perf}}

\def\ev{\on{ev}}

\def\E{ {\EuScript E}}

\def\mod#1{{#1}\on{-mod}}

\newcommand{\co}[1]{ {\langle #1 \rangle}}

\def\presuper#1#2%
  {\mathop{}%
   \mathopen{\vphantom{#2}}^{#1}%
   \kern-\scriptspace%
   #2}
\def\presub#1#2%
  {\mathop{}%
   \mathopen{\vphantom{#2}}_{#1}%
   \kern-\scriptspace%
   #2}

   \def\HH#1{{HH}(#1)}
   \def\NC#1{{HC^{-}}(#1)}
   \def\NCw#1{{HC_{w}^{-}}(#1)}

    \def\Mod#1{{{\rm Mod}}(#1)}

   \def\Spec#1{ {{\on{Spec}(#1)}}}

   \def\Homun#1#2#3{\underline{\on{Hom}}_{#1}(#2,#3)}	
    \def\Endun#1#2{\underline{\on{End}}_{#1}(#2)}
      \def\endun#1{\underline{\on{End}}(#1)}
   \def\End#1#2{{\on{End}}_{#1}(#2)}
   \def\Tan#1{{T}(#1)}

   \def\IndCoh#1{\on{IndCoh}(#1)}
   \def\QCoh#1{\on{QCoh}(#1)}
   \def\Rep#1{\on{Rep}(#1)}
    \def\IndPerf#1{\on{Ind}(\on{Perf}(#1))}
      \def\Loc#1{\on{Loc}(#1)}
    
    \def\DGCat{\on{DGCat_{cont}}}
    \def\DGCattwo{\on{DGCat}^{2}_{\rm cont}}
    \def\dgcat{\on{dgcat}}
    
    \def\Tr#1{\on{Tr}(#1)}
    \def\Trsub#1{\on{Tr}_{#1}}
    \def\Trr#1#2{{\on{Tr}}_{#1}(#2)}
    \def\trr#1{{\on{tr}}_{#1}}
    \def\tr{{\on{tr}}}
    \def\Ind#1{\on{Ind}(#1)}
    
    \def\mult#1{\on{m}_{#1}}
    \def\act#1{\on{m}_{#1}}
   
    \def\res#1#2{\on{res}^{#1}_{#2}}

    \def\Vect{\on{Vect}_{k}}
    \def\Id#1{\on{Id}_{#1}}
    \def\co#1{\on{co}_{#1}}
    \def\ev#1{\on{ev}_{#1}}
    
    \def\unit#1{1_{#1}}
     \def\evL#1{\on{ev}^{l}_{#1}}
      \def\Inv#1{\on{Id}^{!}_{#1}}
    
  \def\F{{\mathcal F}}

   \def\M{{\mathcal M}} 
   
   \def\A{{ A}}
   \def\C{{C}}
   \def\D{{ D}}
   \def\E{{E}}
   
   \def\alg{{\EuScript A}}

   \def\CC{{\EuScript C}}
   \def\DD{{\EuScript D}}

   \def\Funex#1#2{{\on{Fun^{ex}}(#1,#2)}}
   \def\Fun#1#2{{\on{Fun}(#1,#2)}}
    \def\fun{{\on{Fun}}}
   \def\Dual#1#2{\on{\mathbf{D}}_{#1}(#2)}
   
   \def\Coh#1{\on{Coh}(#1)}
   
   \def\Perf#1{\on{Perf}(#1)}
   \def\Hom#1#2#3{{\on{Hom}_{#1}}(#2,#3)}
   \def\hom#1#2{{\on{Hom}}(#1,#2)}
    \def\endo#1{{\on{End}(#1)}}
    \def\bendo#1{{\on{\bf End}(#1)}}
    \def\Corr{{\on{Corr}(\Aff)}}  
    \def\Prstk{\on{PrStk}}
    \def\Prlft{\on{{PrStk}_{laft}}}
    \def\Prlftdef{\on{{PrStk}_{laft-def}}}
  
    \def\Map#1#2{{\on{Map}}(#1,#2)}
    \def\Mapp#1#2#3{{\on{Map}}_{#1}(#2,#3)}
    \def\Gm{\mathbf{G}_{m}}
     \def\Ga{\mathbf{G}_{a}}
    \def\Sym#1#2{{\on{Sym}}_{#1}(#2)}
    \def\sym#1#2{{\on{Sym}}^{#1}(#2)}
    \def\Cotang#1{{T}^{*}(#1)}
   \def\cotang#1#2{{{T}^{*}_{#1}}(#2)}

    \def\Tang#1{{T}(#1)}
   \def\tang#1#2{{{T}_{#1}}(#2)}

   \def\intmod#1#2{{#1}\on{-mod}(#2)}

   \def\calg{\on{CAlg}_{k}} 
   \def\dfor#1#2#3{\mathcal{A}^{#1}(#2,#3)}    
   \def\cldfor#1#2#3{\mathcal{A}^{#1,cl}(#2,#3)}  
    
    \def\Dist#1{\on{Dist}(#1)}
    \def\Prim#1{\on{Prim}(#1)}
    
    \def\vrig#1{(#1)^{\rm vrig}}
    \def\Funop#1#2{\on{Fun}^{\rm oplax}_{\otimes}(#1,#2)}

    \def\Spc{\on{Spc}}
    
    \def\Map#1#2{\on{Map}(#1,#2)}
    \def\Mapun#1#2{\underline{\on{Map}}(#1,#2)}
    
    \def\Funi{{\mathcal F}}
     \def\Funny{\tilde{\mathcal F}}
    \def\Euni{{\EuScript E}}
    
    \def\can#1{\on{can}(#1)}
    
    \newcommand\doubarr{%
        \mathrel{\vcenter{\mathsurround0pt
                \ialign{##\crcr                        
                        \noalign{\nointerlineskip}$\rightarrow$\crcr
                        \noalign{\nointerlineskip}$\rightarrow$\crcr
                }%
        }}%
}

\newcommand{\spmat}[1]{%
  \left(\begin{smallmatrix}#1\end{smallmatrix}\right)%
}

    \newcommand\triparr{%
        \mathrel{\vcenter{\mathsurround0pt
                \ialign{##\crcr
                        \noalign{\nointerlineskip}$\rightarrow$\crcr
                        \noalign{\nointerlineskip}$\rightarrow$\crcr
                        \noalign{\nointerlineskip}$\rightarrow$\crcr
                }%
        }}%
}
    \newcommand{\opeqn}{\begin{equation}}
    \newcommand{\cleqn}{\end{equation}}
    \newcommand{\opeqnn}{\begin{equation*}}
    \newcommand{\cleqnn}{\end{equation*}}

\begin{document}

\title{Relative Calabi-Yau structures II: Shifted Lagrangians in the moduli of objects}
\author{Christopher Brav\footnote{Laboratory for Mirror Symmetry, Higher School of Economics, Moscow, email:{\tt c.brav@hse.ru}}
\;and Tobias Dyckerhoff\footnote{Fachbereich Mathematik, University of Hamburg, email:{\tt tobias.dyckerhoff@uni-hamburg.de}}}

\maketitle
\begin{abstract}
	We show that a Calabi-Yau structure of dimension $d$ on a smooth dg category $\C$
	induces a symplectic form of degree $2-d$ on `the moduli space of objects' $\M_{\C}$.
	We show moreover that a relative Calabi-Yau structure on a dg functor $\C \rightarrow \D$ compatible with
	the absolute Calabi-Yau structure on $C$ induces a Lagrangian structure on the corresponding map of moduli       
	$\M_{\D} \rightarrow \M_{\C}$.
\end{abstract}

\tableofcontents

\vfill\eject

\numberwithin{equation}{section}
\numberwithin{thm}{section}
\setcounter{secnumdepth}{2}

\section{Introduction}

Given a smooth, proper variety $X$ over a field $k$, there is a reasonable derived moduli space of perfect complexes $\M_{X}$ on $X$, with the property that at a point in $\M_{X}$ corresponding to a perfect complex $E$ on $X$, the tangent complex at $E$ identifies with the shifted (derived) endomorphisms of $E$:
\[
\tang{E}{\M_{X}} \simeq \endo{E}[1].
\]
For $X$ of dimension $d$, a trivialisation $\theta: \O_{X} \simeq \Wedge^{d} \Cotang{X}$ of its canonical bundle gives a trace map $\tr : \endo{E} \stackrel{\theta}\simeq \hom{E}{E \otimes \Wedge^{d}\Cotang{X}} \rightarrow k[-d]$ such that the Serre pairing
\begin{equation}\label{spair}
\tang{E}{\M_{X}}[-1]^{\otimes 2} \simeq \endo{E}^{\otimes 2} \stackrel{\circ}\rightarrow \endo{E} \stackrel{\tr}\rightarrow k[-d]
\end{equation}
is anti-symmetric and non-degenerate.

When $d=2$, so that $X$ is a K3 or abelian surface, and the moduli space $\M_{X}$ is replaced with that of simple sheaves, Mukai \cite{mukai} showed that the above pointwise pairings come from a global algebraic symplectic form. Similarly, when $X$ is taken to be a compact oriented topological surface, Goldman \cite{goldman} showed that using Poincar\'e pairings in place of Serre pairings as above gives a global symplectic form on the moduli space of local systems on $X$.

Such examples motivated Pantev-To\"en-Vaqui\'e-Vezzosi \cite{ptvv} to introduce shifted symplectic structures on derived Artin stacks and to show that, in particular, the above pairings \label{spair} are induced by a global symplectic form of degree $2-d$ on $\M_{X}$. The main goal of this paper is to establish an analogue of this global symplectic form when a Calabi-Yau variety $(X,\theta)$ is replaced by a `non-commutative Calabi-Yau' in the form of a nice dg category $\C$ equipped with some extra structure and the moduli space $\M_{X}$ is replaced with a `moduli space of objects' $\M_{\C}$. More precisely, a non-commutative Calabi-Yau of dimension $d$ is a (very) smooth dg category $\C$ equipped with a negative cyclic chain $\theta: k[d] \rightarrow \NC{\C}$ satisfying a certain non-degeneracy condition, and the moduli space $\M_{\C}$ parametrises `pseudo-perfect $\C$-modules', introduced by T\"oen-Vaqui\'e in \cite{tv}. More generally, we shall be interested in `relative left Calabi-Yau structures' on dg functors $\C \rightarrow \D$, in the sense of Brav-Dyckerhoff \cite{cy1}.

The main result of this paper is Theorem \ref{mainthm}, which we paraphrase here.

\begin{prethm}
Given a non-commutative Calabi-Yau $(\C,\theta)$ of dimension $d$, the moduli space of objects $\M_{\C}$ has an induced symplectic form of degree $2-d$. If in addition $f: \C \rightarrow \D$ is a dg functor equipped with a relative left Calabi-Yau structure, then the induced map of moduli spaces $\M_{\D} \rightarrow \M_{\C}$ has an induced Lagrangian structure. 
\end{prethm}

In Corollary \ref{antican}, we shall show that the above theorem about non-commutative Calabi-Yaus allows us to say something new even for non-compact commutative Calabi-Yaus with Gorenstein singularities. Namely, we have the following corollary.

\begin{precor}
Let $X$ be a finite type Gorenstein scheme of dimension $d$ with a trivialisation $\theta : \O_{X} \simeq K_{X}$ of its canonical bundle. Then the moduli space $\M_{X}$ of perfect complexes with proper support has an induced symplectic form of degree $2-d$. When $X$ arises as the zero-scheme of an anticanonical section $s \in K_{Y}^{-1}$ on a Gorenstein scheme $Y$ of dimension $d+1$, then the restriction map
\[
\M_{Y} \rightarrow \M_{X}
\]
carries a Lagrangian structure.
\end{precor}

In Corollary \ref{stackexact}, we shall show that the notion of relative Calabi-Yau structure and its relation to Lagrangian structures allows us to construct Lagrangian correspondences between moduli spaces of quiver representations, generalising examples known to experts. We record here a special case.

\begin{precor} For a noncommutative Calabi-Yau $(\C,\theta)$ of dimension $d$, there is a Lagrangian
correspondence $$\M_{\C} \times \M_{\C} \leftarrow \M^{\rm ex}_{\C} \rightarrow \M_{\C},$$ 
where $\M^{\rm ex}_{\C}$ is the moduli space of exact triangles in $\C$.
\end{precor}

\begin{rem} Before proceeding, let us mention some related work.  The notion of relative Calabi-Yau structure was introduced in our previous paper, \cite{cy1}, where we announced the theorem above. In \cite{toenreview}, 5.3, To\"en sketches an argument for the particular case of the main theorem when $\C$ is both smooth and proper, and describes a version of the second corollary. In \cite{waikit}, Theorem 4.67, Yeung
proves a version of the main theorem for a certain substack of $\M_{\C}$. In \cite{shendetakeda}, Shende and Takeda develop a local-to-global principle for constructing absolute and relative Calabi-Yau structures on dg categories of interest in symplectic topology and representation theory. Combined with our main theorem, this gives many examples of shifted symplectic moduli spaces and Lagrangians in them coming from non-commutative Calabi-Yaus.
\end{rem}

We now sketch the main constructions involved in establishing the main theorem.

First, by definition of the moduli space $\M_{\C}$, there is a universal functor
\[
\Funi_{\C} : \C^{c} \rightarrow \Perf{\M_{\C}}
\]
from the subcategory of compact objects of $\C$ to perfect complexes on the moduli space ${\M_{\C}}$. Applying the functor of Hochschild chains and taking $S^{1}$-invariants, we obtain a map of negative cyclic chains $\NC{\C^{c}} \rightarrow \NC{\Perf{\M_{\C}}}$. An appropriate version of the Hochschild-Kostant-Rosenberg theorem (Proposition \ref{negtoclos}) provides a projection map $\NC{\Perf{\M_{\C}}} \rightarrow \cldfor{2}{\M_{\C}}{2}$
from negative cyclic chains of $\Perf{\M_{\C}}$ to closed $2$-forms of degree $2$. In particular, from a Calabi-Yau structure of dimension $d$, $\theta: k[d] \rightarrow \NC{\C}$, we obtain a closed $2$-form of degree $2-d$ as the composition
\[
k[d] \stackrel{\theta}\rightarrow \NC{\C} \rightarrow \NC{\Perf{\M_{\C}}} \rightarrow \cldfor{2}{\M_{\C}}{2}.
\]

While the construction of the above closed $2$-form is fairly easy, it requires some work to show that it is non-degenerate. Indeed, much of the paper consists in setting up the theory necessary for computing this $2$-form
in such a way that its non-degeneracy becomes manifest. The computation is broken into a number of steps. 

First, we note that since $\C$ is smooth, the functor $\Funi$ is corepresentable relative to $\Perf{\M_{\C}}$ in the sense that there is a universal object $\Euni_{\C} \in \C^{c} \otimes \Perf{\M_{\C}}$ so that 
$\Funi_{\C}=\Homun{\M_{\C}}{\Euni_{\C}}{-}$. Moreover, there is a form of Serre duality relative to $\Perf{\M_{\C}}$, formulated in terms of the `relative inverse dualising functor' (see Corollary \ref{invserre}), which in the case that $(\C,\theta)$ is a noncommutative Calabi-Yau of dimension $d$ induces a global version of the Serre pairing \ref{spair}:

\begin{equation}\label{globalserre}
\Endun{\M_{\C}}{\Euni_{\C}} ^{\otimes 2}\stackrel{\circ}\rightarrow  \Endun{\M_{\C}}{\Euni_{\C}} \stackrel{\tr}\rightarrow \O_{\M_{\C}}[-d].
\end{equation}

Next, we show (see Proposition \ref{tangmod}) that there is a natural isomorphism of Lie algebras of the shifted tangent complex of $\M_{\C}$ with endomorphisms of $\Euni_{\C}$:
\[\label{lieend}
\Tang{\M_{\C}}[-1] \simeq \Endun{\M_{\C}}{\Euni_{\C}}.
\]
In particular, the shifted tangent complex $\Tang{\M_{\C}}[-1]$ carries not only a Lie algebra structure, but even an associative algebra structure. 

Finally, after a general study of maps of Hochschild chains induced by dg functors, we check that under the identification $\Tang{\M_{\C}}[-1] \simeq \Endun{\M_{\C}}{\Euni_{\C}}$, the pairing \ref{globalserre} agrees with that given by the $2$-form induced by $\theta$. (See Proposition \ref{pforms} in the body of the text.)

We end this introduction with an outline of the structure of the paper, highlighting those points important to the proof of the main theorem.

In Section \ref{sectdgcat}, we introduce notation for dg categories. The two most important points 
are Corollary \ref{corep}, which shows that certain dg functors are corepresentable, and Lemma \ref{invserre},
which shows that the `inverse dualising functor' for a smooth dg category behaves like an `inverse Serre functor'.

In Section \ref{modofobjs}, we introduce some basic objects of derived algebraic geometry, as well as the protagonist of our story, the `moduli space of objects' $\M_{\C}$ in a dg category $\C$. The main result of this section is Theorem \ref{tangmod}, which for nice $\C$ establishes an isomorphism of Lie algebras $\Tang{\M_{\C}}[-1] \simeq \endun{\Euni_{\C}}$, where $\Euni_{\C}$ is the `universal left proper object'. In particular, this endows the shifted tangent complex $\Tang{\M_{\C}}[-1]$ with the structure of associative algebra.  

In Section \ref{traces}, we review the formalism of traces of endofunctors, which we use to describe the functoriality and $S^{1}$-action for Hochschild chains. The most import points are Lemma \ref{hochrig}, which describes how to compute the Hochschild map for a dg functor with smooth source and rigid target,  and Proposition \ref{circacts}, which establishes an $S^{1}$-equivariant isomorphism between functions on the loop space $LU$ of an affine scheme $U$ and Hochschild chains $\HH{\QCoh{U}}$ of the category of quasi-coherent sheaves. 

In Section \ref{sympstr}, we review the theory of closed differential forms in derived algebraic geometry. In Proposition \ref{negtoclos}, we show how to construct closed differential forms on the moduli space $\M_{\C}$ from negative cyclic chains on $\C$, and then prove our main result, Theorem \ref{mainthm}. We conclude by discussing some corollaries and examples.

\medskip

\noindent {\bf Conventions}

For ease of reading, we have adopted some linguistic and notational hacks. For example, $(\infty,1)$-categories are simply called categories, $(\infty,1)$-functors are called functors, and homotopy limits and colimits are called limits and colimits. Similarly for $(\infty,2)$-categories. Certain objects or morphisms, such as adjoints and compositions, are only defined up to a contractible space of choices and we leave this ambiguity implicit.
However, given an $(\infty,1)$-category $\C$ and two objects $x,y \in C$, we do write $\Map{x}{y}$ for the mapping space between them, which should serve as a reminder of what is not explicitly mentioned.
Certain properties, like a morphism being an equivalence or an object in a monoidal category being dualisable, can be checked in the homotopy category and we do not usually mention explicitly the passage to the homotopy category. In particular, we simply call equivalences isomorphisms. Since there are no new $\infty$-categorical notions introduced in this paper, and almost all notions that we use appear in standard references such as \cite{lurHT} and \cite{lurHA}, we hope the reader will not have difficulty in applying these conventions.

\medskip

\noindent {\bf Acknowledgements}
We are grateful to Sasha Efimov, Nick Rozenblyum, Artem Prihodko, Pavel Safronov, and Bertrand To\"en for helpful conversations.

\section{Dualisability and smoothness for dg categories}\label{sectdgcat}

In this section we review some basic definitions and results about dg categories. The main results that we use in later sections are Proposition \ref{coreppre} and Corollary \ref{corep}. 

\subsection{Dualisability in symmetric monoidal categories}

In order to aid later calculations, we give a few definitions and make a few observations about dualisable objects and morphisms between them.

We introduce some notation and recall common notions. Let $\CC$ be a symmetric monoidal category. An object $C \in \CC$ is {\bf dualisable} if there is another object $\C^{\vee}$, together with an evaluation $\ev{\C} : \C^{\vee} \otimes \C \rightarrow \unit{\CC}$ and coevaluation $\co{\C} : \unit{\CC} \rightarrow \C \otimes \C^{\vee}$ satisfying the usual axioms. Given a morphism $f : \C \rightarrow \D$ with dualisable source, the {\bf adjoint morphism} $\varphi: 1 \rightarrow \C^{\vee} \otimes \D$ is given as the composition 
\begin{equation}\label{adjmorph}
1 \stackrel{\co{\C}}\longrightarrow \C \otimes \C^{\vee} \simeq \C^{\vee} \otimes \C \stackrel{\Id{\C^{\vee}} \otimes f}\longrightarrow \C^{\vee} \otimes \D.
\end{equation}
Conversely, given a morphism $\varphi: 1 \rightarrow \C^{\vee} \otimes \D$, we obtain the adjoint morphism $f: \C \rightarrow \D$ as the composition
\begin{equation}\label{unadjmorph}
\C \stackrel{\Id{\C} \otimes \varphi}\longrightarrow \C \otimes \C^{\vee} \otimes \D \simeq \C^{\vee} \otimes \C \otimes \D \stackrel{\ev{\C} \otimes \Id{\D}}\longrightarrow \D.
\end{equation}
Note that these two constructions are inverse to each other. Given a morphism $f : \C \rightarrow \D$ with dualisable source and target, the {\bf dual morphism} $f^{\vee} : 
\D^{\vee} \rightarrow \C^{\vee}$ is given as the composition
\begin{equation}\label{dualmorph}
\D^{\vee} \stackrel{\ev{\C}^{\vee} \otimes \Id{\D^{\vee}}}\longrightarrow \C^{\vee} \otimes \C \otimes \D^{\vee}  \stackrel{\Id{\C^{\vee}} \otimes f \otimes \Id{\D^{\vee}}}\longrightarrow \C^{\vee} \otimes \D \otimes \D^{\vee} \stackrel{\Id{\C^{\vee}} \otimes \co{\D}^{\vee}}\longrightarrow \C^{\vee}.
\end{equation}

\begin{rem}
Note that for a dualisable object $\C$, the evaluation $\ev{\C}$ and coevaluation $\co{\C}$ are dual to each other
after composing with the symmetry $\C^{\vee} \otimes \C \simeq \C \otimes \C^{\vee}$. Moreover, the endomorphism of $\C$ adjoint to $\ev{\C}^{\vee} : \unit{\C} \rightarrow \C^{\vee} \otimes \C$ is nothing but the identity endomorphism $\Id{\C}$.
\end{rem}

\begin{lem}\label{dualise} Consider a symmetric monoidal $2$-category $\CC$.

\begin{enumerate}\label{compadj}

\item Let $\C \stackrel{f }\rightarrow \D$ and $\D \stackrel{g}\rightarrow \C$ be morphisms 
between $1$-dualisable objects in $\CC$. Then we have a natural identification of compositions

\begin{equation}\label{coswap}
\Id{\D^{\vee}} \otimes fg \circ \ev{\D}^{\vee} \simeq g^{\vee} \otimes f \circ \ev{\C}^{\vee}
\end{equation} 

In other words, the adjoint of the composition $\D \stackrel{g}\rightarrow \C \stackrel{f}\rightarrow \D$ can be computed as $g^{\vee} \otimes f \circ \ev{\C}^{\vee}$.

\item More generally, given an endomorphism $F : \C \rightarrow \C$ with adjoint morphism $\Phi : \unit{\CC} \rightarrow \C^{\vee} \otimes \C$, the adjoint of the composition $f F g$
can be computed as $g^{\vee} \otimes f \circ \Phi$.

\item Similarly, we have a natural identification
\begin{equation}\label{evswap}
\ev{\D} \circ g^{\vee} \otimes f = \ev{\C} \circ \Id{\C^{\vee}} \otimes gf,
\end{equation}
both sides being adjoint to $gf$.
\item An adjoint pair $f : \C \leftrightarrow \D :f^{r} $ dualises to an adjoint pair $ (f^{r})^{\vee} : \C^{\vee} \leftrightarrow \D^{\vee} : f^{\vee}$. 
\end{enumerate}
\end{lem}

\begin{proof} As these are standard facts, we make only brief remarks on the proofs. 

For 1), using the definition of (co)evaluation,we obtain a factorisation $\Id{\C} \simeq  \co{\C}^{\vee} \otimes \Id{\C} \circ \Id{\C} \otimes \ev{\C}^{\vee}$. Now insert $\Id{\C}$ between
$f$ and $g$, and rearrange, using $g^{\vee} \simeq \Id{\D^{\vee}} \otimes \co{\C}^{\vee} \circ 
\Id{\D^{\vee}} \otimes g \otimes \Id{\C^{\vee}} \circ \ev{\D}^{\vee} \otimes \Id{\C^{\vee}}$.

 For 2), use essentially the same argument as in 1), but replacing $f$ with $Ff$.
 
 For 3), again use the same argument as in 1), but inserting a factorisation of $\Id{\D}$ 
 between $g$ and $f$.
 
 For 4), note that for a $2$-morphism $\alpha: f_{1} \rightarrow f_{2}$, there is a naturally induced $2$-morphism $\alpha^{\vee} : f_{2}^{\vee} \rightarrow f_{1}^{\vee}$. Applying this to the unit and co-unit $f f^{r} \rightarrow \Id{\D}$ and $f^{r} f \rightarrow \Id{\C}$ gives the dualised adjunction.

\end{proof}

\subsection{Presentable dg categories}

In this subsection we discuss the formalism in which we deal with dg categories. Mostly we follow Gaitsgory-Rozenblyum \cite{gr1}. 

$\DGCattwo$ denotes the symmetric monoidal $2$-category of presentable dg categories, continuous dg functors, and dg natural transformations. Here continuous means colimit preserving. The underlying $1$-category, with presentable dg categories as objects and continuous dg functors as $1$-morphisms, is denoted $\DGCat$. We denote by $\fun$ the internal Hom adjoint to tensor product. \footnote{In some sources, $\fun$ is denoted $\fun^{L}$, to emphasise that morphisms preserve colimits.} The unit with respect to the tensor product is the dg category $\Vect$ of dg vector spaces. 

Given a dg category $C \in \DGCat$, we denote its subcategory of compact objects by $C^{c}$. A dg category $C$ is compactly generated if $C=\Ind{C^{c}}$. Note that for any presentable dg category $C$, $C^{c}$ is a small, idempotent complete dg category. The category of such small dg categories is denoted $\dgcat$. 

As a matter of convention, objects of $\DGCat$ shall be called simply `dg categories', while objects of $\dgcat$ shall be called `small dg categories'. Let us emphasise here that in the prequel to this paper \cite{cy1}, we worked with a model for small dg categories $\dgcat$ in terms of small categories enriched over cochain complexes and Morita equivalences between them. In the present paper, it is both more convenient 
and also necessary to work with $\DGCat$, since we to handle not-necessarily compactly generated 
dg categories when dealing with quasi-coherent sheaves on prestacks. 
 
The dualisable objects in  $\DGCat$ ($1$-dualisable objects in $\DGCattwo$) are simply called dualisable dg categories. Concretely, a dg category $\C$ is dualisable if there is another dg category $\C^{\vee}$ and a pairing $\ev{\C} : \C^{\vee} \otimes \C \rightarrow \Vect$ and copairing $\co{\C} : \Vect \rightarrow \C \otimes \C^{\vee}$ satisfying the usual properties. Note that if $\C$ is compactly generated, then it is dualisable with dual $\C^{\vee}=\Ind{(\C^{c})^{op}}$. One shows that $\ev{\C}$ and $\co{\C}$ are dual up to a switch of tensor factors.
Furthermore one shows that for a dualisable dg category, we have a natural equivalence $\C^{\vee} \otimes \D \simeq \Fun{\C}{\D}$, and that under this equivalence, the composition $\Vect \stackrel{\co{\C}}\rightarrow \C \otimes \C^{\vee} \simeq \C^{\vee} \otimes \C \simeq \Fun{\C}{\C}$ sends $k \in \Vect$ to $\Id{\C}$.

Given a continuous dg functor $f : \C \rightarrow \D$ between presentable dg categories (that is, a map in $\DGCat$), the adjoint functor theorem ensures the existence of a formal right adjoint $f^{r} : \D \rightarrow \C$. When the right adjoint $f^{r}$ is itself continuous, we call $f : \C \leftrightarrow \D : f^{r}$ a {\bf continuous adjunction}. When $\C$ and $\D$ are dualisable, passing to duals gives a continuous adjunction $(f^{r})^{\vee}:\C^{\vee} \longleftrightarrow \D^{\vee} : f^{\vee}$, by Lemma \ref{dualise}. One shows that if $\C$ is compactly generated, then a continuous functor $f : \C \rightarrow \D$ has continuous right adjoint if and only $f$ sends compact objects to compact objects.

A dualisable dg category $\C$ is called {\bf proper} if the evaluation functor $\C^{\vee} \otimes \C \stackrel{\ev{\C}}\rightarrow \Vect$ has a continuous right adjoint and is called {\bf smooth} if the evaluation functor has a left adjoint. Equivalently, $\C$ is smooth if the coevaluation functor $ \Vect \stackrel{\co{\C}}\rightarrow \C \otimes \C^{\vee}$ has a continuous right adjoint. Since $\Vect$ is generated by the compact object $k$, $\co{\C}$ has a continuous right adjoint if and only $\co{\C}(k) \in \C \otimes \C^{\vee}$ is compact  if and only if $\Id{\C} \in \Fun{\C}{\C}$ is compact.
(We note in passing that the $2$-dualisable objects in $\DGCattwo$ are precisely the dualisable dg categories $\C$ that are both smooth and proper.)

\subsection{Rigid dg categories and continuous adjunctions}

In this subsection, we review the notion of rigid dg category, following \cite{gr1}, and prove a corepresentability result (Corollary \ref{corep})
for continuous adjunctions between smooth and rigid dg categories. This corepresentability lemma will be important 
for understanding the tangent complex of the moduli space of objects.

By monoidal/symmetric monoidal dg category, we mean an algebra/commutative algebra object in $\DGCat$. 

Given a monoidal dg category $\A$, we denote the tensor product functor by $\mult{\A}: \A \otimes \A \rightarrow \A$, and the unit functor by $-\otimes \unit{\A} : \Vect \rightarrow \A$. Since $\A$ is an algebra object in $\DGCat$, $\mult{\A}$ and $\unit{\A}$ are continuous, hence for every object $a \in \A$, the functors $a \otimes -, - \otimes a : \A \rightarrow \A$ are continuous. 

By $\A$-module category we mean a (left) module $\C$ for $\A$ internal to $\DGCat$. By definition, the action functor $\act{\C} : \A \otimes \C \rightarrow \C$ is continuous. In particular, given any object $c \in \C$, the functor $- \otimes c : \A \rightarrow \C$ is continuous. By the adjoint functor theorem, $- \otimes c$ has a (not necessarily continuous) right adjoint $\Homun{\A}{c}{-} : \C \rightarrow \A$, called `relative Hom'. 

We use the notation 
\[\Endun{\A}{c}:=\Homun{\A}{c}{c}.
\]
$\Endun{\A}{c}$ admits a natural structure of algebra in $\A$. See \cite{lurHA}, 4.7.2.

Given an associative algebra $\alg$ in a monoidal dg category $\A$ and an $\A$-module category $\C$, there is a dg category of $\alg$-modules in $\C$, denoted 
\[
\intmod{\alg}{\C}
\]
The datum of an object $c \in \intmod{\alg}{\C}$ is equivalent to giving an algebra morphism $\alg \rightarrow \Endun{\A}{c}$. 

We shall need the following fact, proved in \cite{gr1}, I.1.8.5.7:

\begin{prop}\label{intmods}
There is an equivalence of categories
\[\mod{\alg} \otimes_\A \C \simeq \intmod{\alg}{\C}.
\]
\end{prop}

A monoidal dg category $\A$ is called {\bf rigid} if the unit $\unit{\A}$ is compact, the monoidal product $\mult{\A} : \A \otimes \A \rightarrow \A$ has a continuous right adjoint $\mult{\A}^{r}$, and $\mult{\A}^{r}$ is a map of $\A$-bimodules. It is easy to see
that $\ev{\A}:=\Hom{\A}{\unit{\A}}{\mult{\A}(-)} : \A \otimes \A \rightarrow \A \rightarrow \Vect$
induces a self-duality equivalence $\A \simeq \A^{\vee}$. When $\A$ is compactly generated, the condition that $\mult{\A}^{r}$ be a bimodule functor can replaced with the requirement that an object is compact if and only if it admits a left and right dual. See  \cite{gr1}, I.1.9.

If $\C$ is dualisable, then one can show that there is an equivalence of dg categories $\C^{\vee}
\simeq \operatorname{Fun}_{\A}{\C}{\A}$ and that there is an $\A$-linear relative evaluation functor
$\ev{\C/\A} : \C^{\vee} \otimes_{\A} \C \rightarrow \A$ exhibiting $\C^{\vee}$ as the $\A$-module dual of $\C$ (\cite{gr1}, I.1.9.5.4). We say that $\C$ is {\bf smooth over $\A$} if the relative evaluation $\ev{\C/\A}$ has a left adjoint $\evL{\C/\A}$ and 
{\bf proper over $\A$} if there is a continuous right adjoint $\ev{\C/\A}^{r}$.

For a rigid dg category $\A$, the induction-restriction adjunction
\begin{equation}
\label{trivfork}
- \otimes 1_{\A}: \Vect \longleftrightarrow \A : \Hom{\A}{1_{\A}}{-}
\end{equation}
is continuous. Tensoring \ref{trivfork} with a dg category $\C$, we obtain a continuous induction-restriction
functor for $\C$ and $\C_{\A}:=\A \otimes \C$, which for brevity we denote 
\[
i : \C \longleftrightarrow \C_{\A} : i^{r}.
\]
Concretely, we have $i(c)= 1_{\A} \otimes c $ and $i^{r}(a \otimes c) = \Hom{\A}{1_{\A}}{a} \otimes c$.

\begin{lem}\label{indadj} Let $\C$ be a dg category, $\A$ a rigid dg category, $f : \C \leftrightarrow \A : f^{r}$  a continuous adjunction. 
\begin{enumerate}
\item There is an induced, continuous $\A$-linear adjunction
\[\xymatrixcolsep{5pc}\xymatrix{F: \C_{\A} \ar@<1ex>[r] |-{\Id{\A} \otimes f} & \A \otimes \A \ar@<1ex>[l] |-{\Id{\A} \otimes f^{r}} \ar@<1ex>[r] |-{\mult{\A}} & \A \ar@<1ex>[l] |-{\mult{\A}^{r}} :F^{r}.}
\]
\item We have $f \simeq F \circ i$ and $f^{r} \simeq i^{r} \circ F^{r}$. Applying $i$ to the latter and using the unit of the adjunction $i,i^{r}$, we obtain a natural transformation 
\begin{equation*}
i \circ f^{r} \simeq i \circ i^{r} \circ F^{r} \Rightarrow F^{r}.
\end{equation*}
\item Using the above natural transformation and the natural isomorphism $i \circ \Phi \simeq \Id{\A} \otimes \Phi \circ i$ for a continuous endomorphism $\Phi$ of $\C$, we obtain a natural transformation 
\begin{equation*}
f \circ \Phi \circ f^{r} \simeq F \circ i \circ \Phi \circ f^{r} \simeq F \circ \Id{\A} \otimes \Phi \circ i \circ f^{r} \Rightarrow
F \circ \Id{\A} \otimes \Phi \circ F^{r}
\end{equation*}
natural in $\Phi$.
\end{enumerate}
\end{lem}

\begin{proof} The proofs are straightforward. Let us merely note that $F$ is  $\A$-linear by construction. The fact that its right adjoint $F^{r}$ is also $\A$-linear uses rigidity of $\A$ and is verified in \cite{gr1}, I.9.3.6.
\end{proof}

Next, we specialise to the case of dualisable and smooth sources and rigid target, where standard diagram chases establish the following. 

\begin{prop}\label{coreppre}
Let $\A$ be a rigid symmetric monoidal dg category, $\C$ a dualisable $\A$-module, $F : \C \leftrightarrow \A: F^{r}$ a continuous $\A$-linear adjunction.

\begin{enumerate}
\item Under the self-duality $\A \simeq \A^{\vee}$, the dual functor $F^{\vee}$ identifies with the composition
\begin{equation*}
\A \stackrel{\ev{\C/\A}^{\vee}}\longrightarrow \C^{\vee} \otimes_{\A} \C \stackrel{\Id{\C^{\vee}} \otimes_{\A} F} \longrightarrow \C^{\vee} \otimes_{\A} \A \simeq \C^{\vee}
\end{equation*}
and the dual functor ${F^{r}}^{\vee}$ identifies with the composition
\begin{equation*}
\C^{\vee} \simeq \A \otimes_{\A} \C^{\vee} \stackrel{F^{r} \otimes_{\A} \Id{\C^{\vee}}}\longrightarrow \C \otimes_{\A} \C^{\vee} \stackrel{\co{\C/\A}^{\vee}}\longrightarrow \A
\end{equation*}

\item By definition of dual functor, $F \simeq \ev{\C/\A} \circ F^{\vee} \otimes_{\A} \Id{\C}$. Then using the above computation of $F^{\vee}$, $F$ identifies with the composition
\begin{equation*}
\C \simeq \A \otimes_{\A} \C \stackrel{\ev{\C/\A}^{\vee} \otimes_{\A} \Id{\C}}\longrightarrow \C^{\vee} \otimes_{\A} \C \otimes_{\A} \C \stackrel{\Id{\C^{\vee}} \otimes_{A} F \otimes_{\A} \Id{\C}}\longrightarrow  \C^{\vee} \otimes_{\A} \A \otimes_{\A} \C \simeq \C^{\vee} \otimes_{\A} \C \stackrel{\ev{\C/\A}}\longrightarrow \A.
\end{equation*}

\item If $\C$ is smooth over $\A$, so that $\ev{\C/\A}: \C^{\vee} \otimes_{\A} \C \rightarrow \A$ has a left adjoint
$\evL{\C/\A}$, then we can pass to left adjoints in $F \simeq \ev{\C/\A} \circ F^{\vee} \otimes_{\A} \Id{\C}$ to obtain
a left adjoint $F^{l} \simeq {F^{r}}^{\vee} \otimes_{\A} \Id{\C} \circ \evL{\C/\A}$. Using the above computation of ${F^{r}}^{\vee}$, we find that $F^{l}$ identifies with the composition
\begin{equation*}
\A \stackrel{\evL{\C/\A}}\longrightarrow \C^{\vee} \otimes_{\A} \C \simeq \A \otimes_{\A} \C^{\vee} \otimes_{\A} \C
\stackrel{F^{r} \otimes_{\A} \Id{\C^{\vee}} \otimes_{\A} \Id{\C}}\rightarrow \C \otimes_{\A} \C^{\vee} \otimes_{\A} \C \stackrel{\co{\C/\A}^{\vee} \otimes_{A} \Id{\C}}\rightarrow \C.
\end{equation*}
Inspecting the above composition, we find that
\begin{equation*}
F^{l} \simeq \Id{\C/\A}^{!} \circ F^{r},
\end{equation*}
where $\Id{\C/\A}^{!}$ is adjoint to $\evL{\C/\A}(k) \in \C^{\vee}\otimes_{A} \C$. 

\item When $\C$ is smooth over $\A$, we set $E := F^{l}(1_{\A}) \in \C$ and obtain that $F$ is corepresentable relative to $\A$:
\begin{equation*}
F \simeq \Homun{\A}{E}{-}.
\end{equation*}
\end{enumerate}

\end{prop}

Let $\A$ be a rigid, compactly generated dg category, $\C$ a compactly generated $\A$-module category. An  object $c \in \C$ is called {\bf left proper} over $\A$ if $\Homun{\A}{c}{-}: \C \rightarrow \A$ is continuous with continuous right adjoint, and {\bf right proper} over $\A$ if $\Homun{\A}{-}{c}^{\vee} : \C \rightarrow \A$ is continuous with continuous right adjoint. \footnote{Here, $\Homun{\A}{-}{c}^{\vee}$ is a slight abuse of notation. Strictly speaking, the formula is correct on compact objects, and is then defined everywhere by left Kan extension.}

The functor $\Id{\C/\A}^{!}$ adjoint to $\evL{\C/\A}(k) \in \C^{\vee}\otimes_{A} \C$ is called the {\bf (relative) inverse dualising functor}, since by the following corollary it behaves like an `inverse Serre functor' relative to $\A$.

\begin{cor}\label{invserre}
Let $\C$ be a compactly generated dg category, smooth over a rigid dg category $\A$. Suppose $c \in \C$ is right proper over $\A$, so that the functor $\Homun{\A}{-}{c}^{\vee} : \C \rightarrow \A$ is continuous with continuous right adjoint. Then there is a natural isomorphism of functors
\[
\Homun{\A}{-}{c}^{\vee} \simeq \Homun{\A}{\Inv{\C/\A}(c)}{-}.
\]
In particular, $\Inv{\C/\A}(c)$ is left proper. 

Moreover, applying the above isomorphism to $c$, we have $\Homun{\A}{\Inv{\C/\A}(c)}{c} \simeq \Homun{\A}{c}{c}^{\vee}$. Composing with the dual of the unit $\unit{\A} \rightarrow \Homun{\A}{c}{c}$, we obtain a trace map $\trr{c} : \Homun{\A}{\Inv{\C/\A}(c)}{c} \simeq \Homun{\A}{c}{c}^{\vee} \rightarrow \unit{\A}$. For a compact object $d$, the isomorphism
$\Homun{\A}{d}{c}^{\vee} \simeq \Homun{\A}{\Inv{\C/\A}(c)}{d}$ is induced by the pairing
\[\Homun{\A}{d}{c} \otimes_{\A} \Homun{\A}{\Inv{\C/\A}(c)}{d} \stackrel{\circ}\rightarrow \Homun{\A}{\Inv{\C/\A}(c)}{c} \stackrel{\trr{c}}\rightarrow \unit{\A}. 
\]
\end{cor}

\begin{proof}
Let $F=\Homun{\A}{-}{c}^{\vee} : \C \rightarrow \A$. By assumption, $F$ has a continuous right adjoint $F^{r}$. For each compact object $d \in \C$, we have a natural equivalence 
\[
\Hom{\A}{d}{c} \simeq \Hom{\A}{\unit{\A}}{\Homun{\A}{d}{c}} \simeq \Hom{\A}{\Homun{\A}{d}{c}^{\vee}}{\unit{\A}}\simeq \Hom{\A}{F(d)}{\unit{\A}} \simeq \Hom{\A}{d}{F^{r}(\unit{\A})},\] hence by the Yoneda lemma $F^{r}(\unit{\A}) \simeq c$. By Proposition \ref{coreppre}, $F$ also has a left adjoint given as $F^{l}=\Inv{\C/\A} \circ F^{r}$ and $F$ is corepresented by $F^{l}(\unit{\A})$, hence $F \simeq \Homun{\A}{\Inv{\C/\A}(c)}{-}$, as claimed.

The statement about the isomorphism being induced by the pairing follows from naturality of the isomorphism, just as in the case of Serre functors.
\end{proof}

Combining Lemma \ref{indadj} and Proposition \ref{coreppre}, we have the following corepresentability result, which will be essential in understanding the tangent complex of the moduli space of objects $\M_{\C}$ in a smooth dg category $\C$.

\begin{cor}\label{corep}
Let $f : \C \longleftrightarrow \A : f^{r}$ be a continuous adjunction with smooth source and rigid target. Then the induced functor
\[F=f_{\A} : \C_{\A} \rightarrow \A 
\]
has a left adjoint $F^{l}$ and $F$ is corepresented by the compact object $E=F^{l}(\unit{\A}) \in \C_{\A}$:
\[
F \simeq \Homun{\A}{E}{-}.
\]
We have isomorphisms
\begin{align}\label{endisos}
& FF^{l}(\unit{\A}) \simeq F \Inv{\C/\A} F^{r}(\unit{\A}) \simeq  \Endun{\A}{E}  \\
& FF^{r}(\unit{\A}) \simeq \Homun{\A}{E}{F^{r}(\unit{\A})} \simeq  \Endun{\A}{E}^{\vee}  
\end{align}
\end{cor}

We end this section with a computation that will be useful later for computing fibres of certain canonical perfect complexes on the moduli space of objects in a dg category.

\begin{lem}\label{homfibre}
Let $\C$ be a dg category, $\A$ a rigid dg category, and $\varphi : \A \longleftrightarrow \Vect: \varphi^{r}$ an adjunction with $\varphi$ a symmetric monoidal dg functor. Then for objects $E_{1},E_{2} \in \C_{\A}=\A \otimes \C$,
we have a natural isomorphism
\[
\varphi \Homun{\A}{E_{1}}{E_{2}} \simeq \Hom{\C}{(\varphi \otimes \Id{\C})(E_{1})}{(\varphi \otimes \Id{\C})(E_{2})}.
\]
\end{lem}

\begin{proof}
First, let us note that $\Vect$ becomes an $\A$-module via $\varphi$ and that with respect to this $A$-module structure
$\varphi^{r}$ is $\A$-linear. Hence the endofunctor $\varphi^{r} \varphi$ of $\A$ is $\A$-linear and so determined by its action on $\unit{\A}$, giving an isomorphism of functors
\[\varphi^{r} \varphi \simeq \varphi^{r}(k) \otimes -.
\]
Using this isomorphism, adjunction, and $A$-linearity of internal Hom, we obtain the following sequence of isomorphisms:

\begin{gather*}
\varphi \Homun{\A}{E_{1}}{E_{2}} \simeq \Hom{k}{\varphi(\unit{A})}{\varphi \Homun{\A}{E_{1}}{E_{2}} } \simeq \Hom{\A}{\unit{\A}}{\varphi^{r} \varphi \Homun{\A}{E_{1}}{E_{2}}} \simeq \\
\Hom{\A}{\unit{\A}}{\varphi^{r}(k) \otimes \Homun{\A}{E_{1}}{E_{2}}} \simeq \Hom{\A}{\unit{\A}}{\Homun{\A}{E_{1}}{\varphi^{r}(k) \otimes E_{2}}} \simeq \\ 
\Hom{\C_{\A}}{E_{1}}{\varphi^{r}(k) \otimes E_{2}} \simeq  \Hom{\C}{(\varphi \otimes \Id{\C})(E_{1})}{(\varphi \otimes \Id{\C})(E_{2})}.
\end{gather*}

\end{proof}

\section{The moduli space of objects}\label{modofobjs}

\subsection{Quasi-coherent and ind-coherent sheaves on affine schemes}\label{qcaff}

We review some basic notions in derived algebraic geometry that we shall need later, mostly following \cite{gr1}, Chapters 2-6. For more subtle points, we give precise references.

For now on, we take $k$ be a field of characteristic $0$.

By definition, the category of (derived) {\bf affine schemes} $\Aff$ is opposite to the category $\calg^{\leq 0} \subset \calg$ of connective commutative algebras in $\Vect$. \footnote{Since we are working in characteristic $0$, it is possible to model $\calg^{\leq 0}$ in terms of cohomologically non-positive commutative differential graded algebras. See \cite{lurHA}, Proposition 7.1.4.11.} 

An affine scheme $U=\Spec{R}$ is said to be of {\bf finite type} over the ground field $k$ if $H^{0}(R)$ is finitely generated as a commutative algebra over $k$, $H^{i}(R)$ is finitely generated as a module over $H^{0}(R)$, and $H^{-i}(R)=0$ for $i>>0$. The category of affine schemes of finite type is denoted $\Aft$.

By definition, the dg-category of {\bf quasi-coherent sheaves} $\QCoh{U}$ on an affine scheme $U=\Spec{R}$ is the 
dg category of dg modules over the commutative algebra $R$. Given a map $f: U \rightarrow V$, the pullback functor $f^{*} : \QCoh{V} \rightarrow \QCoh{U}$ is given by induction of modules along the corresponding map of rings. As such, $f^{*}$ is symmetric monoidal. The naturality of pullback is expressed via a functor
\[
\QCoh{-}^{*} : \Aff^{\rm op} \rightarrow \DGCat
\]

Since we are so far considering only affine schemes, $f^{*}$ always has a continuous right adjoint $f_{*}$. 

One can show that $\QCoh{U}$ is a rigid symmetric monoidal dg category, and in particular that $\otimes$-dualisable objects coincide with compact objects. In this case, the structure sheaf $\O_{U}$, corresponding to the ring $R$, is a compact generator. The compact objects in $\QCoh{U}$ are called {\bf perfect complexes}, which form a small idempotent complete dg category denoted $\Perf{U}$. 
They are preserved by pullback. In the present affine case, we therefore have $\Ind{\Perf{U}}=\QCoh{U}$.

Given a pullback square of affine schemes
\begin{equation}\label{pullback}
\xymatrix{
U \times_{W} V \ar[d]^{q} \ar[r]^{p} & V \ar[d]^{g} \\
U \ar[r]^{f} & W
}
\end{equation}
naturality of pullback gives an isomorphism $q^{*}f^{*} \simeq p^{*}g^{*}$, so by adjunction we obtain a {\bf base-change map} 
\begin{equation}\label{basechange}
f^{*}g_{*} \rightarrow q_{*}p^{*},
\end{equation}
which is easily checked to be an isomorphism by considering its action on the generator $\O_{V} \in \QCoh{V}$. 

For affine schemes $U=\Spec{R}$ of finite type, define the small subcategory $\Coh{U} \subseteq \QCoh{U}$ of {\bf coherent sheaves} to consist of quasi-coherent sheaves with bounded, finitely generated cohomology: $F \in \Coh{U}$ if $H^{i}(F)$ is finitely generated over $H^{0}(R)$ and $H^{i}(F)=0$ for $|i|>>0$. The dg category of {\bf ind-coherent sheaves} is defined to be the ind-completion of the category of coherent sheaves:
\[\IndCoh{U}:= \Ind{\Coh{U}}.
\]
The category of ind-coherent sheaves $\IndCoh{U}$ is a module category for quasi-coherent sheaves $\QCoh{U}$,
with the action given by ind-completion of the action of $\Perf{U}$ on $\Coh{U}$. For a map of affine schemes of finite type $f : U \rightarrow V$, there is a functor \footnote{For an `elementary' definition of $f^{!}$, see \cite{gr1}, II.5.4.3.}
\[f^{!} : \IndCoh{V} \rightarrow \IndCoh{U}.
\]
More precisely, we have a functor
\[
\IndCoh{}^{!} : \Aft^{\rm op} \rightarrow \DGCat,
\]

$\IndCoh{U}$ has a natural symmetric monoidal structure, the product of which is denoted $\otimes^{!}$, and the unit of which is $\omega_{U}:=p^{!}(k)$ for $p: U \rightarrow *$. Using the action of $\QCoh{U}$ on $\IndCoh{U}$, tensoring with $\omega_{U}$ gives a symmetric monoidal functor 
\[
\Upsilon_{U}={-}\otimes \omega_{U} : \QCoh{U} \rightarrow \IndCoh{U}.
\]
The functor $\Upsilon$ intertwines $*$-pullback and $!$-pullback: $\Upsilon_{U}f^{*} \simeq f^{!} \Upsilon_{V}$.

More precisely, $\Upsilon$ is a natural transformation
\[
\Upsilon: \QCoh{-}^{*} \rightarrow \IndCoh{-}^{!}
\]
of functors from $\Aft^{\rm op}$ to $\DGCat$.

There is a self-duality equivalence $\IndCoh{U} \simeq \IndCoh{U}^{\vee}$. The corresponding equivalence between compact objects is denoted 
\[
\Dual{U}{-}=\Coh{U}^{\rm op} \simeq \Coh{U}
\]
One can show that there is an isomorphism of functors $\Dual{U}{-} \simeq \Homun{\QCoh{U}}{-}{\omega_{U}}$. The functor $\Dual{U}{-}$ can be used to define a contravariant {\bf Grothendieck-Serre duality} functor
\begin{equation}\label{contradual}
{\QCoh{U}^{-}}^{op} \rightarrow \Funex{\Coh{U}^{op}}{\Vect} \simeq \IndCoh{U}
\end{equation}
given explicitly by $E \mapsto \Hom{\QCoh{U}}{E}{\Dual{U}{-}}$ \footnote{Here, $\QCoh{U}^{-}$ denotes
quasi-coherent sheaves that are cohomologically bounded above. For more on Grothendieck-Serre duality, see \cite{gr2}, I.1.3.4.}. If $E$ is a perfect complex, then for any $F \in \Coh{U} \subset \IndCoh{U}$, we have isomorphisms
\[
\Hom{\QCoh{U}^{-}}{E}{\Dual{U}{F}} \simeq
\Hom{\IndCoh{U}}{E \otimes F}{\omega_{U}} \simeq \Hom{\IndCoh{U}}{F}{E^{\vee} \otimes \omega_{U}}
\]
hence the functor \ref{contradual} is given on perfect complexes by 
\begin{equation}\label{perfdual}
E \mapsto \Upsilon(E^{\vee}).
\end{equation} 
In particular, it is symmetric monoidal and fully faithful when restricted to perfect complexes. More generally,
one can show that $\Dual{U}{-}$ is fully faithful on bounded above quasi-coherent sheaves having coherent cohomology sheaves.










\subsection{Prestacks and the moduli of objects}

In this subsection, we fix notation by reviewing some basic constructions concerning prestacks and dg categories of sheaves on prestacks. Our basic reference is \cite{gr1},\cite{gr2}.

We denote by $\Prstk := \Fun{\calg^{\leq 0}}{\Spc}$ the category of prestacks on $\Aff$. Being a topos, $\Prstk$ is cocomplete, Cartesian closed, and colimits commute with pullbacks. We denote the internal/local mapping space adjoint to $X \times -$ by $\Mapun{X}{-}$, and the global mapping space by $\Map{X}{-}$. Moreover, there is a continuous faithful embedding $\Spc \hookrightarrow \Prstk$ sending a space $K$ to the constant prestack with value $K$. 

The embedding $\Spc \hookrightarrow \Prstk$ is symmetric monoidal for the Cartesian monoidal structures, so (abelian) groups in $\Spc$ map to (abelian) groups in $\Prstk$. We shall be especially interested in the circle group $B\ZZ=S^{1}$. 

\begin{defi}
Given a prestack $X$, its free loop space $LX$ is by definition the mapping prestack $\Mapun{S^{1}}{X}$. 
\end{defi}

The free loop space $LX$ carries a natural action of the circle group $S^{1}$, which we call `loop rotation'. Decomposing a circle into two intervals and using the fact that mapping out of a colimit gives a limit, we obtain an
isomorphism of the free loop space with the self-intersection of the diagonal:
\[
LX \simeq X \times_{X \times X} X
\]
In particular, if $X$ is affine, then the free loop space is again affine.

Mostly we shall be interested in prestacks that are {\it laft} (locally almost of finite type) and {\it def} (`have deformation theory'). Roughly, a prestack $X$ is {\bf laft} if it is determined by maps $U \rightarrow X$ with $U$ an affine of finite type, and is {\bf def} if it has a (pro-)cotangent complex $\Cotang{X}$ that behaves as expected. See the next section for what we expect of a (pro-)cotangent complex. 

Recall from Section \ref{qcaff} the functor of quasi-coherent sheaves on affine schemes:
\[
 \QCoh{-}^{*} : \Aff^{\rm op} \rightarrow \DGCat
\]

Taking the right Kan extension of $\QCoh{-}^{*}$, we obtain a functorial notion of quasi-coherent sheaves on general prestacks:
\[
\QCoh{-}^{*}:\Prstk^{\rm op} \rightarrow \DGCat.
\]
Since every prestack $X$ is tautologically a colimit over all affines mapping into it, $X=\colim_{\Aff/X} U$, we have by definition an identification
\[
\QCoh{X} = \lim_{(\Aff/X)^{op}} \QCoh{U}.
\]

For each map of prestacks $f : X \rightarrow Y$, we have by definition a pullback functor $f^{*} : \QCoh{Y} \rightarrow \QCoh{X}$. The adjoint functor theorem provides a right adjoint, denoted $f_*$, but in general
it can be poorly behaved. However, for `qca' morphisms $f$, $f_*$ is continuous and satisfies base change and the projection formula for pullbacks along maps of affines $U \rightarrow Y$ (see Corollary 1.4.5 \cite{dringaits}). A morphism $f: X \rightarrow Y$ is qca if the pullback of $X$ along a map from any affine $U \rightarrow Y$ is a nice Artin $1$-stack with affine stabilisers. This will be obvious in the situations where we need it.

One can similarly define perfect complexes on a prestack by right Kan extension from affines, so that in particular we have an identification
\[
\Perf{X} = \lim_{(\Aff/X)^{op}} \Perf{U}.
\]
For a general prestack $X$, perfect complexes need not be compact as objects in $\QCoh{X}$, but they always identify with the subcategory of $\otimes$-dualisable objects in $\QCoh{X}$. In particular, $\QCoh{X}$ is not always rigid, nor even dualisable in $\DGCat$. It shall therefore be convenient for us to formally introduce the category of {\bf ind-perfect sheaves} $\Ind{\Perf{X}}$. Note that by construction $\Ind{\Perf{X}}$ is compactly generated and that pullback preserves compact objects, hence for a map of prestacks $f : X \rightarrow Y$, we have a continuous adjunction
\[
f^{*}: \Ind{\Perf{Y}} \longleftrightarrow \Ind{\Perf{X}}: f_{*}
\]

Similarly, for a general laft prestack $X$, the category of ind-coherent sheaves is defined as the limit along $!$-pullback over all finite type affine schemes mapping to $X$:
\[
\IndCoh{X}:= \lim_{(\Aft/X)^{\rm op}} \IndCoh{U}^{!}.
\]
For a map of laft prestacks $f : X \rightarrow Y$, we have an evident pullback functor $f^{!} : \IndCoh{Y} \rightarrow \IndCoh{X}$ and a natural transformation $\Upsilon : \QCoh{-}^{*} \rightarrow \IndCoh{-}^{!}$ of functors from $\Prlft^{\rm op}$ to $\DGCat$ given at a laft-prestack $X$ by tensoring with $\omega_{X}$.

\begin{rem}
For maps of laft prestacks $f : X \rightarrow Y$ that are sufficiently algebraic, one can define a pushforward functor $f_{*} : \IndCoh{X} \rightarrow \IndCoh{Y}$. Beware, however, that unless $f$ is proper, $f^{!}$ is {\it not} right adjoint to $f_{*}$. Nonetheless, one of the main results of \cite{gr1},\cite{gr2} is that $*$-pushforward satisfies base-change with respect to $!$-pullback.
\end{rem}

We can now define the main object of interest for this paper.  

\begin{ex}
The {\bf moduli space of objects} $\M_{\C}$ in a compactly generated dg category $\C$ is the prestack given on an affine $U$ by
\[
\M_{\C}(U)=\Mapp{\dgcat}{\C^{c}}{\Perf{U}}.
\]
Note that $\Mapp{\dgcat}{\C^{c}}{\Perf{U}}$ is the space of exact functors $\C^{c} \rightarrow \Perf{U}$ from compact objects in $\C$ to perfect complexes on $U$. Equivalently, we could consider the space of continuous adjunctions $\C \leftrightarrow \QCoh{U}$. 

When $\C$ is smooth, Corollary \ref{corep} ensures that functors $F: \C^{c} \rightarrow \Perf{U}$ are precisely those 
co-represented by left proper objects $E \in \C^{c} \otimes \Perf{U}$, hence the (somewhat inaccurate) name `moduli space of objects'. In particular, a $k$-point $x : \Spec{k} \rightarrow \M_{\C}$ classifies a functor
\[
\varphi_{x}: \C^{c} \rightarrow \Perf{k},
\] 
and when $\C$ is smooth, this functor is corepresented by $\Inv{\C}\varphi^{r}_{x}(k)$. By Serre duality, we have $\Hom{\C}{\Inv{\C}\varphi^{r}_{x}(k)}{y} \simeq \Hom{\C}{y}{\varphi^{r}_{x}(k)}^{*}$ naturally in compact objects $y \in \C^{c}$, hence we have an isomorphism of functors $\varphi_{x} \simeq \Hom{\C}{-}{\varphi^{r}_{x}{k}}^{*}$. Our convention is to identify the point $x$ with the right proper object $\varphi^{r}(k)$, so that we have an isomorphism of functors
\begin{equation}\label{kpoint}
\varphi_{x} \simeq \Hom{\C}{-}{x}^{*}.
\end{equation}

By definition of the moduli space, there is a universal exact functor $\C^{c} \rightarrow \Perf{\M_{\C}}$, or equivalently, a universal continuous adjunction
\[
\Funi_{\C} : \C \longleftrightarrow \Ind{\Perf{\M_{\C}}} : \Funi^{r}_{\C}, 
\] 
so that given a continuous adjunction $F : \C \longleftrightarrow \QCoh{U} : F^{r}$ corresponding
to a morphism $f : U \rightarrow \M_{\C}$, we have an isomorphism
\[
f^{*}\Funi_{\C} \simeq F.
\]

By Corollary \ref{corep}, the universal functor $\Funi_{\C}$ is corepresented by a left proper object 
\[
\Euni_{\C} \in \Ind{\Perf{\M_{U}}} \otimes \C.
\] 
\end{ex}

\begin{rem}
The moduli space $\M_{\C}$ was introduced by To\"en-Vaqui\'e \cite{tv}, where it is shown that for $\C$ a finite type dg category, $\M_{\C}$ is locally an Artin stack of finite presentation and in particular has a perfect cotangent complex. A compactly generated dg category $\C$ is of {\bf finite type} if its category of compact objects $\C^{c}$ is compact in the category $\dgcat$ of small idempotent complete dg categories and exact functors. One can show that finite type dg categories are always smooth. See \cite{tv}, Proposition 2.14.
\end{rem}

\subsection{(Co)tangent complexes and differential forms}

In this subsection, we review the notions of cotangent complex and tangent complex, following I.1 of \cite{gr2}. 
(In fact, \cite{gr2} work with the somewhat more general notion of pro-cotangent complex, but we shall not explicitly need that.)

Given an affine scheme $U=\Spec{R}$ and a connective quasi-coherent sheaf $F \in \QCoh{U}^{\leq 0}$,
we form the trivial square-zero extension $U_{F}=\Spec{R \oplus F}$. Given a prestack $X$ and a point
$U \stackrel{x}\rightarrow X$, the {\bf space of derivations} at $x$ valued in $F$ is by definition
\[
\Mapp{U/}{U_{F}}{X}.
\]
For a fixed point $x$, the space of derivations valued in $F$ is natural in $F$ and we obtain a functor
\[\label{deratpt}
\QCoh{U}^{\leq 0} \rightarrow \Spc, F \mapsto \Mapp{U/}{U_{F}}{X}.
\]
When this functor respects fibres of maps $F_1 \rightarrow F_2$ inducing surjections on $H^0$, it can be extended to an exact functor 
\begin{equation}
\label{derminpt}
\QCoh{U}^{-} \rightarrow \Spc, F \mapsto \Mapp{U/}{U_{F}}{X}.
\end{equation}
We say that $X$ has a {\bf cotangent space} $\cotang{x}{X} \in \QCoh{U}^{-}$ at $U \stackrel{x}\rightarrow X$ if the functor \ref{derminpt} is 
corepresented by  $\cotang{x}{X}$:
\[
\Mapp{\QCoh{U}^{-}}{\cotang{x}{X}}{F} \simeq \Mapp{U/}{U_{F}}{X}.
\]

Suppose $X$ has all cotangent spaces and
\begin{equation}
\label{affsover}
\xymatrix{
U \ar[rr]^{f} \ar[dr]^{x} & & V \ar[dl]^{y} \\
& X}
\end{equation}
is a commutative diagram of affines over $X$. Then there is a natural pullback map
\begin{equation}
\label{pullcotan}
f^{*}\cotang{y}{X} \rightarrow \cotang{x}{X}.
\end{equation}
If \ref{pullcotan} is an isomorphism for all diagrams \ref{affsover}, we obtain a {\bf cotangent complex}
\[
\Cotang{X} \in \QCoh{X}= \lim_{(\Aff/X)^{\rm op}} \QCoh{U}
\]
whose fibres are the cotangent spaces:
\[
x^{*}\Cotang{X} \simeq \cotang{x}{X} \in \QCoh{U}^{-}.
\]

Similarly, given a map of prestacks $X \rightarrow Y$ and a point $x : U \rightarrow X$, the functor of {\bf relative derivations} at $x$ is
\begin{equation}
\label{relder}
F \mapsto \Mapp{U/}{U_{F}}{X} \times_{\Mapp{U/}{U_{F}}{Y}} *.
\end{equation}
If the functor \ref{relder} is co-represented by an object $\cotang{x}{X/Y} \in \QCoh{U}^{-}$, the co-representing object
$\cotang{x}{X/Y}$ is called the relative cotangent space at $x$, and if relative cotangent spaces at different points are compatible under pullback, then we obtain a {\bf relative cotangent complex} $\Cotang{X/Y} \in \QCoh{X}$.

\begin{rem}
One can show in particular that filtered colimits of Artin stacks have cotangent complexes, and that Artin stacks locally of finite presentation have perfect cotangent complexes. In particular, the moduli space $\M_{\C}$ for a finite type dg category $\C$ has a perfect cotangent complex. See \cite{tv}, Theorem 3.6.
\end{rem}

Given a laft prestack $X$ with cotangent complex $\Cotang{X}$, its {\it tangent complex}
\[
\Tan{X} \in \IndCoh{X}
\]
is defined to be the image of its cotangent complex under the contravariant duality \ref{contradual}. In particular,
when the cotangent complex of $X$ is perfect, we have by \ref{perfdual} an identification
\[
\Upsilon(\Cotang{X}^{\vee}) \simeq \Tan{X}
\]
We define the complex of {\it differential $p$-forms} on $X$ to be
\[
\Wedge^{p}\Cotang{X} \in \QCoh{X}.
\]
and the space of differential $p$-forms of degree $n$ to be
\[
\dfor{p}{X}{n}=|\Gamma(X,\Wedge^{p}\Cotang{X} [n])|. \footnote{Here $|| : \Vect \rightarrow \Spc$ is the `geometric realisation' of a complex, which is truncation above at $0$ followed by the Dold-Kan correspondence.}
\]
When $\Cotang{X}$ is perfect, we have by \ref{perfdual} isomorphisms
\begin{equation}\label{dualforms}
\Gamma(X,\Wedge^{p}\Cotang{X}[n]) \simeq \Hom{\QCoh{X}}{\O_{X}}{\Wedge^{p}\Cotang{X}[n]} \simeq \Hom{\IndCoh{X}}{\Wedge^{p}\Tan{X}[-n]}{\omega_{X}}.
\end{equation}

\subsection{The tangent complex of the moduli of objects}

In this subsection, we compute the shifted tangent complex of the moduli of objects $\Tan{\M_{\C}}[-1]$ in a finite type dg category $\C$. Our argument is an adaptation of that of \cite{gr2}, II.8.3.3, which treats the case $\C=\Vect$.

To begin with, we review the  construction of the natural Lie algebra structure on $\Tan{X}[-1] \in \IndCoh{X}$ for $X \in \Prlftdef$. 

Given $X \in \Prlftdef$, consider the completion $(X \times X)^{\wedge}$ of the diagonal $\Delta: X \rightarrow X \times X$ as a pointed formal moduli problem over $X$:
\[
\xymatrix{(X \times X)^{\wedge} \ar[d]^{p_{s}} \\
X \ar@/^1pc/[u]^{\Delta}}
\] 
Looping, we obtain a formal group $\Omega_{X}(X \times X)^{\wedge}$ over $X$ sitting in a pullback diagram
\[
\xymatrix{\Omega_{X}(X \times X)^{\wedge} \ar[d]^{\pi} \ar[r]^{\pi} & X \ar[d]^{\Delta} \\
X \ar[r]^{\Delta} & (X \times X)^{\wedge}
}
\]
It is easy to check that the formal group $\Omega_{X}(X \times X)^{\wedge}$ identifies with the completion $LX^{\wedge}$ of the loop space along the constant loops.

From the theory of formal groups developed in \cite{gr2}, II.7.3, $LX^{\wedge}$ has a cocommutative Hopf algebra of distributions in $\IndCoh{X}$ given as 
\[
\Dist{LX^{\wedge}}=\pi_{*}\omega_{LX^{\wedge}} \simeq \Delta^{!}\Delta_{*}\omega_{X},
\]
whose Lie algebra of `primitive elements' identifies with the shifted tangent complex:
\[
\Prim{\pi_{*}\omega_{LX^{\wedge}} } \simeq \Tan{X}[-1].
\]
By \cite{gr2}, II.6.1.7, there is an isomorphism of cocommutative conilpotent coalgebras
\begin{equation}\label{pbw}
\Sym{X}{\Tang{X}[-1]} \simeq \Delta^{!}\Delta_{*}\omega_{X}.
\end{equation}
By \cite{gr2}, II.7.5.2 and II.8.6.1, there is a natural identification
\begin{equation}
\label{reps}
\Delta^{!}\Delta_{*}\omega_{X} -{\rm mod}(\IndCoh{X}) \simeq \IndCoh{(X \times X)^{\wedge}}
\end{equation}
where the functor $\Delta^{!} : \IndCoh{(X \times X)^{\wedge}} \rightarrow \IndCoh{X}$ corresponds to the forgetful functor and $p_{1}^{!} : \IndCoh{X} \rightarrow \IndCoh{(X \times X)^{\wedge}}$ to the trivial module functor. Taking $!$-pullback along the other factor gives another symmetric monoidal functor $\can{X}: p_{2}^{!} : \IndCoh{X} \rightarrow \IndCoh{(X \times X)^{\wedge}}$. In particular, there is an action map $\Delta^{!}\Delta_{*}\omega_{X} \otimes^{!} F \rightarrow F$ natural in $F \in \IndCoh{X}$, and hence by adjunction an algebra map $\Delta^{!}\Delta_{*}\omega_{X} \rightarrow \Endun{X}{F}$.
In particular, for a perfect complex $\Upsilon(E) \in \IndCoh{X}$, we have an algebra map
\begin{equation}\label{algact}
\Delta^{!}\Delta_{*}\omega_{X} \rightarrow \Endun{X}{\Upsilon(E)} \simeq \Upsilon \Endun{X}{E}.
\end{equation}

\begin{rem}
One can show that for a perfect complex $\Upsilon(E) \in \IndCoh{X}$, the corresponding action map $\Tang{X}[-1] \otimes^{!} \Upsilon(E) \rightarrow \Delta^{!}\Delta_{*}\omega_{X} \otimes^{!} \Upsilon(E) \rightarrow \Upsilon(E)$ identifies with the Atiyah class of $E$. Compare \cite{henn}.
\end{rem}

We shall need the following, which combines the equivalences of \ref{reps} and Proposition \ref{intmods}.

\begin{prop}\label{envmods}
Let $M$ be a module category for $\IndCoh{X}$. Then there is an equivalence 
\begin{equation*}
\IndCoh{(X \times X)^{\wedge}} \otimes_{\IndCoh{X}} M \simeq \intmod{\Delta^{!}\Delta_{*}\omega_{X}}{M}
\end{equation*}
Tensoring over $\IndCoh{X}$ with the functor $\can{X} \simeq p_{2}^{!}: \IndCoh{X} \rightarrow \IndCoh{(X \times X)^{\wedge}} \simeq \intmod{\Delta^{!}\Delta_{*}\omega_{X}}{\IndCoh{X}}$, we obtain a functor $\can{M} : M \rightarrow \intmod{\Delta^{!}\Delta_{*}\omega_{X}}{M}$, endowing every object $m \in M$ with a canonical structure
of $\Delta^{!}\Delta_{*}\omega_{X}$-module. In particular, we obtain a canonical action map $\Delta^{!}\Delta_{*}\omega_{X} \otimes m \rightarrow m$. Adjoint to this, we obtain a natural algebra map 
\begin{equation}\label{algact}
\Delta^{!}\Delta_{*}\omega_{X} \rightarrow \Endun{X}{m}
\end{equation}
in $\IndCoh{X}$.
\end{prop}

For later use, we elaborate on a particular case of the above proposition.

\begin{lem}\label{deltadual}
Let $\C$ be a smooth dg category, $f: \C \rightarrow \QCoh{U}$ a continuous functor with continuous right adjoint $f^{r}$, where $U$ is an affine scheme of finite type, and $E \in \C_{U}=\QCoh{U} \otimes \C $ the object
corepresenting $F=f \otimes \Id{U}$, so that $F \simeq \Homun{U}{E}{-}$. Then the map $\Delta^{!}\Delta_{*}\omega_{U} \rightarrow \Upsilon \Endun{U}{E}$ in $\IndCoh{U}$ from \ref{algact} is Grothendieck-Serre dual to the natural map $\Endun{U}{E}^{\vee} \rightarrow \Delta^{*}\Delta_{*}\O_{U}$ in $\QCoh{U}$.
\end{lem}

\begin{proof}
The assertion is clear at the level of objects. Indeed, since $\Endun{U}{E}$ is perfect, $\Dual{U}{\Endun{U}{E}^{\vee}} \simeq \Endun{U}{E} \otimes \omega_{U} \simeq \Upsilon \Endun{U}{E}$. Moreover, by definition of the duality functor $\Dual{U}{-}$ \ref{contradual}, we have 
\begin{eqnarray*}
\Hom{\IndCoh{U}}{F}{\Dual{U}{\Delta^{*}\Delta_{*}\O_{U}}} \simeq \Hom{\QCoh{U}}{\Delta^{*}\Delta_{*}\O_{U}}{\Dual{U}{F}} \simeq 
\Hom{\QCoh{U \times U}}{\Delta_{*}\O_{U}}{\Dual{U \times U}{\Delta_{*}F}} \simeq \\ \Hom{\IndCoh{U \times U}}{{\Delta_{*}F}}{\Delta_{*}\omega_{U}} \simeq \Hom{\IndCoh{U}}{F}{\Delta^{!}\Delta_{*}\omega_{U}}
\end{eqnarray*}
for $F \in \Coh{U}$, hence
by the Yoneda lemma $\Dual{U}{\Delta^{*}\Delta_{*}\O_{U}}$ and $\Delta^{!}\Delta_{*}\omega_{U}$ are naturally isomorphic.

At the level of morphisms, writing  $\Upsilon\Endun{U}{E} \simeq \Delta^{!}p^{!}_{2} \Upsilon \Endun{U}{E}$,
we have that the map $\Delta^{!}\Delta_{*}\omega_{U} \rightarrow \Upsilon \Endun{U}{E}$ is obtained by applying $\Delta^{!}$ to the natural map $\Delta_{*}\omega_{U} \rightarrow p^{!}_{2}\Upsilon \Endun{U}{E}$ adjoint to the unit $\omega_{U} \rightarrow \Delta^{!}p^{!}_{2}\Upsilon \Endun{U}{E} \simeq \Upsilon \Endun{U}{E}$. Similarly,
writing $\Endun{U}{E}^{\vee}  \simeq \Delta^{*}p^{*}_2 \Endun{U}{E}^{\vee}$, the natural map $\Endun{U}{E}^{\vee} \rightarrow \Delta^{*}\Delta_{*}\O_{U}$ is obtained by applying $\Delta^{*}$ to the natural map $p^{*}_2 \Endun{U}{E}^{\vee} \rightarrow \Delta_{*}\O_{U}$ adjoint to the map $\Endun{U}{E}^{\vee} \rightarrow \O_{U}$
dual to the unit $\O_{U} \rightarrow \Endun{U}{E}$. Since the duality functor $\mathbf{D}$ exchanges $*$-pullback and $!$-pullback, the assertion follows.
\end{proof}

We now proceed to compute the shifted tangent complex of the moduli of objects $\M_{\C}$ in a dg category 
$\C$ of finite type.

Recall that by definition we have a universal continuous adjunction
\[
\Funi_{\C}: \C \longleftrightarrow \IndPerf{\M_{\C}} : \Funi^{r}_{\C}
\]
and hence by Corollary \ref{corep}, there is a left proper object
\[
\Euni_{\C} \in \IndPerf{\M_{\C}} \otimes \C
\]
corepresenting $\Funi_{\C}$. In particular, we obtain an associative algebra $\Endun{\M_{\C}}{\Euni_{\C}}$ in $\Perf{\M_{\C}}$ and hence an associative algebra $\Upsilon \Endun{\M_{\C}}{\Euni_{\C}} \simeq \Endun{\M_{\C}}{\Upsilon \Euni_{\C}}$ in $\IndCoh{\M_{\C}}$. 

Using Proposition \ref{envmods} with $M=\IndCoh{\M_{\C}} \otimes \C$ and $m=\Upsilon \Euni_{\C}$, 
we obtain a natural map of algebras
\[
\Delta^{!}\Delta_{*}\omega_{\M_{\C}} \rightarrow \Endun{\M_{\C}}{\Upsilon \Euni_{\C}}
\]
and hence a map of Lie algebras
\begin{equation}\label{liemap}
\Tang{\M_{\C}}[-1] \rightarrow \Endun{\M_{\C}}{\Upsilon \Euni_{\C}}.
\end{equation}

\begin{prop}\label{tangmod}
The map of Lie algebras \eqref{liemap} is an isomorphism.
\end{prop}

\begin{proof}
Given a point $x: U \rightarrow \M_{\C}$ classifying a functor $f: \C \rightarrow \QCoh{U}$, let $E \in \QCoh{U} \otimes \C$ be the left proper object corepresenting the functor $f \otimes \Id{U} : \C_{U}=\QCoh{U} \otimes \C \rightarrow \QCoh{U}$. Applying $!$-pullback to \eqref{liemap}, we obtain for every $\F \in \Coh{U}$ a map
\begin{equation}\label{pulledliemap}
\Hom{U}{\Dual{U}{F}}{x^{!}\Tang{\M_{\C}}[-1]} \rightarrow  \Hom{U}{\Dual{U}{F}}{\Endun{U}{\Upsilon E}}.
\end{equation}
Since $\M_{\C}$ and $L\M_{\C}$ are laft-def, to show that \eqref{liemap} is an isomorphism it suffices
to check that \eqref{pulledliemap} is an isomorphism for all $\F \in \Coh{U}$, and since $\IndCoh{U}$ is stable, 
it is in fact enough to check that \eqref{pulledliemap} induces an isomorphism on homotopy classes of maps.
We shall do this by showing that $x^{!}\Tang{\M_{\C}}[-1]$ and $\Endun{U}{\Upsilon E}$ represent the
same functor at the level of homotopy categories.

By definition of the Lie algebra of a formal group, we have $\Tan{\M_{\C}}[-1] \simeq s^{!}\Tan{L\M_{\C}/\M_{\C}}$,
hence represents relative derivations for $L\M_{\C} \rightarrow \M_{\C}$ at each point $U \rightarrow \M_{\C} \stackrel{s}\rightarrow L\M_{\C}$.  By definition of $\M_{\C}$ and of the loop space, a point $U \rightarrow \M_{\C} \stackrel{s}\rightarrow L\M_{\C}$ is given by a pair $(f,\Id{f})$, where $f$ is a functor $f : \C \rightarrow \QCoh{U}$ and $\Id{f}$ is the identity automorphism of the functor $f$. Therefore to give a relative derivation into $F \in \Coh{U}$ is to give an automorphism $\alpha$ of the trivial extension $\tilde{f}: \C \stackrel{f}\rightarrow \QCoh{U} \stackrel{\pi^{*}}\rightarrow \QCoh{U_{F}}$ of the functor $f$ together with an identification $i^{*}\alpha \simeq \Id{f}$.

By adjunction, the automorphism $\alpha : \pi^{*}f \rightarrow \pi^{*}f$ is equivalent to a map $E \rightarrow \pi_{*}\pi^{*}E \simeq E \oplus F \otimes E$ in $\C_{U}$ whose first component is just $\Id{E}$. Such a map is therefore determined by its second component $E \rightarrow F \otimes E$.
In short, homotopy classes of derivations with values in $F \in \Coh{U}$ at $(f,\Id{f}): U \rightarrow L\M_{\C}$ relative to $\M_{\C}$ naturally identify with homotopy classes of maps $E \rightarrow F \otimes E$
in $\C_{U}$.

We claim that such  maps are naturally identified with maps $\Dual{U}{F} \rightarrow \Upsilon \Endun{U}{E}$, and thus $\Upsilon \Endun{U}{E}$ identifies with the relative tangent space for every point. Indeed, we have

\begin{align*}
& \Hom{\C_{U}}{E}{F \otimes E}  \simeq \Hom{\QCoh{U}}{\O_{U}}{\Homun{\QCoh{U}}{E}{F \otimes E}} \simeq \\
& \Hom{\QCoh{U}}{\O_{U}}{F \otimes \Endun{\QCoh{U}}{E}} \simeq  \Hom{\QCoh{U}}{\Endun{\QCoh{U}}{\E}^{\vee}}{F} \simeq \\ & \Hom{\IndCoh{U}}{\Dual{U}{F}}{\Dual{U}{{\Endun{\QCoh{U}}{\E}^{\vee}}}} \simeq \Hom{\IndCoh{U}}{\Dual{U}{F}}{\Upsilon \Endun{\QCoh{U}}{E}}.
\end{align*}

\end{proof}

We conclude this section with a computation of the (co)tangent map induced by a dg functor.

\begin{lem}\label{tangmap}
Let $f : \C \longleftrightarrow \D: f^{r}$ be a continuous adjunction between smooth dg categories and $\varphi: \M_{\D} \rightarrow \M_{\C}$ the induced map of moduli spaces. Then there is a natural map
of functors $\Funi_{\D}\Funi^{l}_{\D} \varphi^{*} \rightarrow \varphi^{*}\Funi_{\C}\Funi^{l}_{\C}$
which when evaluated on $\O_{\M_{\C}}$ gives a map $\Endun{\M_{\D}}{\Euni_{\D}} \rightarrow \varphi^{*}\Endun{\M_{\C}}{\Euni_{\C}}$. After applying $\Upsilon$, the latter map gives the shifted tangent map
\[
\Tang{\M_{\D}}[-1] \rightarrow \varphi^{!}\Tang{\M_{\C}}[-1].
\]
The fibre of the above shifted tangent map at a point $x \in \M_{\D}$ corresponding to a functor $\varphi_{x}=\Hom{\D}{-}{x}^{*} : \D \rightarrow \Vect$ identifies with the map $\End{\D}{x} \rightarrow \End{\C}{f^{r}(x)}$ induced by the functor $f^{r} : \D \rightarrow \C$.

Dually, there is a natural map of functors $\varphi^{*}\Funi_{C}\Funi^{r}_{\C} \rightarrow \Funi_{\D}\Funi^{r}_{\D}\varphi^{*}$ which when evaluated on $\O_{\M_{\C}}$ gives a map $\varphi^{*}\Endun{\M_{\C}}{\Euni_{\C}}^{\vee} \rightarrow \Endun{\M_{\D}}{\Euni_{\D}}$. The latter map identifies with the shifted cotangent map
\[
\varphi^{*}\Cotang{\M_{\C}}[1] \rightarrow \Cotang{\M_{\D}}[1].
\]
\end{lem}

\begin{proof}
The universal property of the moduli spaces gives a commutative diagram of functors
\[
\xymatrix{\C \ar[r]^{f} \ar[d]^{\Funi_{\C}} & \D \ar[d]^{\Funi_{\D}} \\
\Ind{\Perf{\M_{\C}}} \ar[r]^{\varphi^{*}} & \Ind{\Perf{\M_{\D}}} }.
\]
We then have a composition of natural maps of functors $
\Funi^{l}_{\D}\varphi^{*} \rightarrow \Funi^{l}_{\D} \varphi^{*}\Funi_{\C}\Funi^{l}_{\C} \simeq \Funi^{l}_{\D}\Funi_{\D}f\Funi^{l}_{\C} \rightarrow f\Funi^{l}_{\C}$ where the first arrow is induced by the unit $\Id{\C} \rightarrow \Funi_{\C} \Funi^{l}_{\C}$ and the second by the counit $\Funi^{l}_{\D}\Funi_{\D} \rightarrow \Id{\D}$. Applying $\Funi_{\D}$ to this composition gives the desired map $\Funi_{\D}\Funi^{l}_{\D}\varphi^{*} \rightarrow \Funi_{\D}f\Funi^{l}_{\C} \simeq \varphi^{*}\Funi_{\C} \Funi^{l}_{\C}$. Evaluating on $\O_{\M_{\C}}$ indeed gives a map
$\Endun{\M_{\D}}{\Euni_{\D}} \rightarrow \varphi^{*}\Endun{\M_{\C}}{\Euni_{\C}}$ by Corollary \ref{corep}. Using Proposition \ref{tangmod} and applying $\Upsilon$, we obtain a map $\Tang{\M_{\D}}[-1] \rightarrow \varphi^{!}\Tang{\M_{\C}}[-1]$. That this map agrees with the natural tangent map follows easily from the same kind of argument as in the proof of Proposition \ref{tangmod}. Finally, the claim about the fibres follows from Lemma \ref{homfibre}.

The dual statement for the cotangent map is proved dually.

\end{proof}

\section{Traces and Hochschild chains}\label{traces}

\subsection{Traces and circle actions}\label{subtraces}


We begin by reviewing the theory of traces in (higher) symmetric monoidal categories. Our main reference is Hoyois-Scherotzke-Sibilla \cite{hss}, which among other things provides enhanced functoriality for a construction of To\"en-Vezzosi \cite{tvchern}. Other references making use of this circle of ideas include \cite{bzntraces} and \cite{kondprih}. We follow \cite{hss}, but slightly modify the notation and language to be consistent with other parts of the paper. In particular, we call a symmetric monoidal category `very rigid' rather than `rigid' if all its objects are dualisable. 

Following \cite{hss}, given a symmetric monoidal $2$-category $\CC$, we consider the symmetric monoidal $1$-category $\bendo{\CC}$, defined as the symmetric monoidal category of `oplax natural transfors', in the sense of Scheimbauer-Johnson-Freyd \cite{sjf}, from the free very rigid category generated $B\NN$ to $\CC$:
\begin{equation}
\bendo{\CC}:=\Funop{\vrig{B\NN}}{{\CC}}.
\end{equation}
Accordingly, we shall informally say that that $\bendo{\CC}$ is `oplax corepresentable'. At the level of homotopy categories,  $\bendo{\CC}$ admits the following description: an object of $\bendo{\CC}$ is a pair $(\C,\Phi)$, where $\C \in \CC$ is a $1$-dualisable object and $\Phi$ is an endomorphism
of $x$. Given two objects $(\C,\Phi)$ and $(\D,\Psi)$, a morphism between them is a pair $(f,\alpha)$, where
$f : \C \rightarrow \D$ is a $1$-morphism admitting a right adjoint $f^{r}$ in $\CC$ and $\alpha : f \Phi \Rightarrow \Psi f$ is a $2$-morphism. Such a morphism is usually displayed as a lax commutative square

\begin{equation}\label{laxsq}
\xymatrix{\C \ar[d]^{f} \ar[r]^{\Phi} & \C \ar[d]^{f} \ar@{=>}[dl]^{\alpha}\\
\D \ar[r]^{\Psi} & \D}
\end{equation}
The symmetric monoidal structure on $\bendo{\CC}$ is given `pointwise'. We also consider the symmetric monoidal category $\Omega \CC$, whose objects are endomorphisms of the unit $\unit{\CC}$ and whose morphisms are natural transformations between such endomorphisms. 

Definitions 2.9 and 2.11 of \cite{hss} give a symmetric monoidal trace functor 
\begin{equation}\label{endtr}
{\rm Tr}: \bendo{\CC} \rightarrow \Omega \CC.
\end{equation}
The value of ${\rm Tr}$ on an object $(C,\Phi)$ is computed simply as the trace of the endomorphism adjoint to $\Phi$, namely, as the composition $\unit{\CC} \stackrel{\ev{\C}^{\vee}}\longrightarrow \C^{\vee} \otimes \C \stackrel{\Id{\C^{\vee}} \otimes \Phi} \longrightarrow \C^{\vee} \otimes \C \stackrel{\ev{\C}}\longrightarrow \unit{\CC}$. In other words, the trace of $\Phi$ is the composition of the morphism $\unit{\CC} \stackrel{\Phi^{\rm ad}}\rightarrow \C^{\vee} \otimes \C$ adjoint to $\C \stackrel{\Phi}\rightarrow \C$ with the evaluation morphism $\ev{\C} : \C^{\vee} \otimes \C \rightarrow \unit{\CC}$:

\begin{equation}\label{tracephi}
\Tr{\Phi}=\ev{\C}(\Phi^{\rm ad}).
\end{equation}
Given a morphism $(f,\alpha) : (\C,\Phi) \rightarrow (\D,\Psi)$ in $\bendo{\CC}$, the induced map of traces $\Tr{\Phi} \Rightarrow \Tr{\Psi}$ is computed as the left-to-right composition of lax-commutative squares
\begin{equation}
\xymatrixcolsep{2.5pc}
\xymatrixrowsep{3.5pc}
\xymatrix{ 1   \ar@{=}[d] \ar[r]^-{\ev{\C}^{\vee}} & \C^{\vee} \otimes \C \ar[d]|-{(f^{r})^{\vee} \otimes f } \ar[r]^{\Id{\C} \otimes \Phi} \ar@{=>}[dl] & \C^{\vee} \otimes \C \ar[d]|-{(f^{r})^{\vee} \otimes f } \ar[r]^-{\ev{\C}}    \ar@{=>}[dl]  & 1 \ar@{=}[d]  \ar@{=>}[dl] \\
 1 \ar[r]^-{\ev{\D}^{\vee}} & \D^{\vee} \otimes \D \ar[r]^{\Id{\D} \otimes \Psi} & \D^{\vee} \otimes \D \ar[r]^-{\ev{\D}} & 1  }.
\end{equation}
Here, we have used Lemma \ref{dualise} to define the $2$-morphisms in the left-most and right-most squares
as $(f^{r})^{\vee} \otimes f \circ \ev{\C}^{\vee}\simeq (\Id{\D^{\vee}} \otimes ff^{r}) \circ \ev{\D}^{\vee}  \stackrel{(1 \otimes \varepsilon) \circ \ev{\D}^{\vee}}\longrightarrow \ev{\D}^{\vee}$ and $\ev{\C} \stackrel{\ev{\C} \circ (1 \otimes \eta)}\longrightarrow \ev{\C} \circ (\Id{\C^{\vee}} \otimes f^{r}f) \simeq \ev{\D} \circ (f^{r})^{\vee} \otimes f$, while the $2$-morphism in the central square is $1 \otimes \alpha$.

\begin{lem}\label{tracefact}
Given a morphism $(f,\alpha):(\C,\Phi) \rightarrow (\D,\Psi)$ corresponding to a lax commutative square \ref{laxsq}, the induced map of traces $\Tr{f,\alpha}:\Tr{\Phi} \rightarrow \Tr{\Psi}$ factors as
\begin{equation*}
\Tr{\Phi} \stackrel{\Tr{\Phi \eta }}\longrightarrow \Tr{\Phi f^{r}f} \simeq \Tr{f\Phi f^{r}} \stackrel{\Tr{\alpha f^{r}}}\longrightarrow \Tr{\Psi ff^{r}}
\stackrel{\Tr{\Psi \varepsilon}}\longrightarrow \Tr{\Psi}
\end{equation*}
\end{lem}

\begin{proof}
Observe that the diagram
\[
\xymatrix{\C \ar[d]^{f} \ar[r]^{\Phi} & \C \ar[d]^{f} \ar@{=>}[dl]^{\alpha}\\
\D \ar[r]^{\Psi} & \D}
\]
factors as

\[
\xymatrixrowsep{3.5pc}
\xymatrixcolsep{2.75pc}
\xymatrix{\C \ar@{=}[d] \ar[r]^{\Phi} & \C \ar@{=}[d] \ar@{=>}[dl]^{\Phi \eta}\\
\C \ar[d]^{f} \ar[r]^{\Phi ff^{r}} & \C \ar[d]^{f} \ar@{=}[dl] \\
\D \ar@{=}[d] \ar[r]^{f \Phi f^{r}} & \D\ar@{=}[d] \ar@{=>}[dl]^{\alpha f^{r}} \\
\D \ar@{=}[d] \ar[r]^{\Psi ff^{r}} & \D \ar@{=}[d] \ar@{=>}[dl]^{\Psi \varepsilon} \\
\D \ar[r]^{\Psi} & \D}
\]

\end{proof}

An important feature of the theory of traces developed in \cite{hss} is the naturality in $\CC$ of the trace functor
$\Trsub{\CC}: \bendo{\CC} \rightarrow \Omega \CC$. While not explicitly stated in \cite{hss}, the following lemma follows immediately from `oplax corepresentability' of $\bendo{\CC}$. 

\begin{lem}\label{trnat}
Given a symmetric monoidal 2-functor $F: \CC \rightarrow \DD$, we have a commutative diagram of symmetric monoidal $2$-functors
\[
\xymatrix{
\bendo{\CC} \ar[r]^{\Trsub{\CC}} \ar[d]^{F} & \Omega \CC \ar[d]^{F} \\
\bendo{\DD} \ar[r]^{\Trsub{\DD}} & \Omega \DD}
\]
Explicitly, given an object $(\C,\Phi) \in \CC$, we have an equivalence
\[F(\Trr{\CC}{\Phi}) \simeq \Trr{\DD}{F(\Phi)}.
\]
Furthermore, if $G$ is right adjoint to $F$, then for any object $(\D,\Psi)$ in $\D$, the counit $F \circ G \Rightarrow \Id{\D}$ induces a natural map
$F(\Trr{\CC}{G \Psi}) \simeq \Trr{\DD}{FG \Psi} \rightarrow \Trr{\DD}{\Psi}$ and hence, by adjunction, a natural map
\[
\Trr{\CC}{G \Psi} \rightarrow G \Trr{\DD}{\Psi}.
\]
\end{lem}

Similarly to the category of endomorphisms $\bendo{\CC}$, we define the category of automorphisms
as
\[
\bAut{\CC}:=\Funop{\vrig{S^{1}}}{\CC}.
\]
At the level of homotopy categories, $\bAut{\CC}$ admits the following description. The objects of $\bAut{\CC}$ are pairs $(\C,\Phi)$ of a dualisable object in $\CC$ together with an automorphism $\Phi$. The $1$-morphisms in $\bAut{\CC}$ are the same as those in $\bendo{\CC}$. Restricting along $B\NN \rightarrow B\ZZ=S^{1}$, we obtain a symmetric monoidal trace functor
\begin{equation}
\Trr{\CC}{-} :\bAut{\CC} \rightarrow \Omega \CC.
\end{equation}

The main result that we need from \cite{hss} is Theorem 2.14 (refining Corollaire 2.19 of \cite{tvchern}), which states that the trace functor $\Trr{\CC}{-} : \bAut{\CC} \rightarrow \Omega \CC$ admits a unique $S^{1}$-equivariant lift natural in symmetric monoidal functors $\CC \rightarrow \DD$. Here $\bAut{\CC}=\Funop{\vrig{S^{1}}}{\CC}$ carries the $S^{1}$-action induced by that on $\vrig{S^{1}}$, while $\Omega \CC$ carries the trivial $S^{1}$-action.
Here we explicitly formulate the result from \cite{hss} that we shall need later.

\begin{prop}\label{circleeq}
Given an $S^{1}$-fixed point $(\C,\Phi) \in \bAut{\CC}$, there is an induced $S^{1}$-fixed point structure on $\Trr{\CC}{\Phi} \in \Omega \CC$, that is, an $S^{1}$ action on $\Trr{\CC}{\Phi}$. Given a second $S^{1}$-fixed point $(\D,\Psi) \in \bAut{\CC}$, and an $S^{1}$-fixed map $(f,\alpha) : (\C,\Phi) \rightarrow (\D,\Psi)$, we get an induced $S^{1}$-equivariant map $\Trr{\CC}{\Phi} \rightarrow \Trr{\DD}{\Psi}$.

Moreover, given a symmetric monoidal functor $F: \CC \rightarrow \DD$ between symmetric monoidal $2$-categories, 
we obtain an $S^{1}$-equivariant equivalence 
\begin{equation}
F(\Trr{\CC}{\Phi}) \simeq \Trr{\DD}{F(\Phi)}
\end{equation}
\end{prop}

The case of most interest to us will be the trace of the identity functor $\Id{\C}$ on a dualisable object $\C \in \CC$, which is naturally $S^{1}$-fixed. In the next subsection, we consider the special case of the symmetric monoidal $2$-category of presentable dg categories, in which case $\Tr{\Id{\C}}$ gives a natural realisation of Hochschild chains of $\C$ with its functorial $S^{1}$-action. In the following subsection, we consider the special case of the symmetric monoidal $2$-category of correspondences of affine (derived) schemes, and use Proposition \ref{circleeq} to identify Hochschild chains and functions on the loop space as $S^{1}$-complexes.

\begin{rem}
While the constructions above were described mostly at the level of homotopy categories, which is sufficient for later computations, the {\it existence} of a homotopy coherent trace functor and its $S^{1}$-equivariant lift
are important for us and provided by \cite{hss} and \cite{tvchern}. As we have briefly indicated, homotopy coherence and functoriality are handled by defining the symmetric monoidal categories $\bendo{\CC}$ and $\bAut{\CC}$ to be `oplax corepresentable' by $\vrig{B\NN}$ and $\vrig{B\ZZ}$ respectively.  
\end{rem}

\subsection{Hochschild chains of dg categories}

We now specialise to the case of the symmetric monoidal $2$-category $\DGCattwo$ of presentable dg categories. Given a dualisable dg category $\C \in \DGCattwo$, we define {\bf Hochschild chains} of $\C$ to be trace of the identity functor on $\C$ endowed with the $S^{1}$-action described in the last section:

\[
\HH{\C}:= \Tr{\Id{\C}}
\]

\begin{rem}
There are various approaches in the literature to the $S^{1}$-action on Hochschild chains. Most classically, the $S^{1}$-action is described in terms of the cyclic bar complex, as in the book of Loday \cite{loday}. Comparable to this is the construction of Hochschild chains in terms of factorisation homology, as in \cite{lurHA} and \cite{amr}. In this paper we use the $S^{1}$-action coming from the cobordism hypothesis, as in \cite{tvchern}. While the comparison between the first two $S^{1}$-actions and the third seem to be known to experts, we so far have not found a reference. Nonetheless, we have chosen not to reflect this ambiguity in the notation.
\end{rem}

Given a continuous adjunction $f : \C \longleftrightarrow \D: f^{r}$ between dualisable dg categories, we obtain
from the formalism of traces an induced $S^{1}$-equivariant map
\[\HH{\C} \rightarrow \HH{\D}.\]

Recall from Section \ref{sectdgcat} that when $\C$ is smooth, then by definition the evaluation functor $\ev{\C}: \C^{\vee} \otimes \C \rightarrow \Vect$
has a left adjoint $\evL{\C} : \Vect \rightarrow \C^{\vee} \otimes \C$. Under the identification $\C ^{\vee} \otimes \C \simeq \endo{\C}$, $\evL{\C}(k)$ corresponds to a continuous endofunctor of $\C$, denoted $\Id{\C}^{!}$ and called the inverse dualising functor of $\C$. By definition of the identification $\C^{\vee} \otimes \C \simeq \endo{\C}$, the action of $\Id{\C}^{!}$ is given by the composition

\begin{equation}\label{invdual}
\Id{\C}^{!} : \C \stackrel{\Id{\C} \otimes \evL{\C}}\longrightarrow \C \otimes \C^{\vee} \otimes \C \stackrel{\tau \otimes \Id{\C}}\simeq \C^{\vee} \otimes \C \otimes \C  \stackrel{\ev{\C} \otimes \Id{\C}}\longrightarrow \C
\end{equation}

Forgetting the $S^{1}$-action, we obtain the following expression for Hochschild chains of a smooth dg category $\C$ in terms of ${\rm Hom}$-complexes:

\begin{equation}\label{hochsmooth}
\HH{\C}=\Tr{\Id{\C}}=\Hom{k}{k}{\ev{\C} \circ \ev{\C}^{\vee}(k)} \simeq \Hom{\C^{\vee} \otimes \C} {\evL{\C}(k)}{\ev{\C}^{\vee}(k)} \simeq \Hom{\endo{\C}}{\Id{\C}^{!}}{\Id{\C}}.
\end{equation}

Using the above identification, we can compute the map on Hochschild chains for a dualisable functor with smooth source and dualisable target and in particular for smooth source and smooth target.

\begin{prop}\label{hochfact}
Let $f: \C \leftrightarrow \D: f^{r}$ be a continuous adjunction with smooth source and dualisable target. Given a Hochschild chain $k[i] \rightarrow \Tr{\Id{\C}}$ adjoint to a natural transformation $\alpha : \Inv{\C}[i] \rightarrow \Id{\C}$, the composition $k[i] \rightarrow \Tr{\Id{\C}} \rightarrow \Tr{\Id{\D}}$ giving the image of the Hochschild chain under the functor $f$ identifies with the composition
\begin{equation*}
k[i] \rightarrow \Tr{\Inv{\C}}[i] \stackrel{\Tr{\Inv{\C}}\eta}[i]\longrightarrow \Tr{\Inv{\C}f^{r}f}[i]\simeq \Tr{f \Inv{\C}f^{r}}[i] \stackrel{\Tr{f \alpha f^{r}}}\longrightarrow \Tr{ff^{r}} \stackrel{\Tr{\varepsilon}}\longrightarrow \Tr{\Id{\D}}. 
\end{equation*}

When $\D$ is also smooth, there is a natural unit map $\tilde{\eta}: \Inv{\D} \rightarrow f \Inv{\C} f^{r}$ so that the image of $\alpha$ identifies with the composition
\begin{equation}
\Inv{\D}[i] \stackrel{\tilde{\eta}[i]}\longrightarrow  f \Inv{\C}[i] f^{r} \stackrel{f \alpha f^{r}}\longrightarrow ff^{r}\stackrel{\varepsilon}\longrightarrow \Id{\D}
\end{equation}
under the isomorphism $\Tr{\Id{\D}} \simeq \Hom{\endo{\D}}{\Inv{\D}}{\Id{\D}}$.

\end{prop}

\begin{proof}
First note that $\Tr{\Inv{\C}}=\ev{\C} \circ \evL{\C}(k)$, so there is a natural unit $k \rightarrow \Tr{\Inv{\C}}$. After suspension, that gives the first arrow. Then by adjunction, the composition $k[i] \rightarrow \Tr{\Inv{\C}}[i] \stackrel{\Tr{\alpha}}\rightarrow \Tr{\Id{\C}}$ identifies with the original Hochschild chain $k[i] \rightarrow \Tr{\Id{\C}}$. Now using Lemma \ref{tracefact}, and the naturality of $\eta : \Id{\C} \rightarrow f^{r}f$, we obtain the commutative diagram
\begin{equation*}
\xymatrix{k[i] \ar[r] & \Tr{\Inv{\C}[i]} \ar[r] \ar[d]& \Tr{\Id{\C}} \ar[d] \\
                              & \Tr{ \Inv{\C}f^{r}f}[i] \ar@{}[d]|*=0[@]{\simeq} \ar[r] & \Tr{f^{r}f} \ar@{}[d]|*=0[@]{\simeq} \\
                              & \Tr{f \Inv{\C} f^{r}}[i] \ar[r] & \Tr{ff^{r}} \ar[r] & \Tr{\Id{\D}}}
\end{equation*}

Now suppose both $\C$ and $\D$ are smooth. Since they are in particular dualisable, we have a natural transformation $\ev{\C} \rightarrow \ev{\D} \circ (f^{r})^{\vee} \otimes f$.
Applying $\evL{\D}$ on the left and $\evL{\C}$ on the right of this natural transformation, we obtain
a map $\evL{\D} \circ \ev{\C} \circ \evL{\C} \rightarrow \evL{\D} \circ \ev{\D} \circ (f^{r}) \otimes f \circ \evL{\C}$.
Post-composing with the counit $\evL{\D} \circ \ev{\D} \rightarrow \Id{\D}$, we obtain a map 
\begin{equation}\label{sechalf}
\evL{\D} \circ \ev{\C} \circ \evL{\C} \rightarrow (f^{r}) \otimes f \circ \evL{\C}
\end{equation} 
Since $\C$ is smooth, we have a unit $k \rightarrow \ev{\C} \circ \evL{\C}(k)=\endo{\Inv{\C}}$. Applying
$\evL{\D}$ on the left of this unit, we obtain a map 
\begin{equation}\label{frsthalf}
\evL{\D}(k) \rightarrow  \evL{\D} \circ \ev{\C} \circ \evL{\C}(k).
\end{equation} 
Composing \ref{sechalf} and \ref{frsthalf} and using the usual identifications, we obtain the desired unit
\begin{equation}\label{shriekunit}
\Inv{\D} \stackrel{\tilde{\eta}}\longrightarrow f \Inv{\C} f^{r}.
\end{equation}
The claim about the image of $\alpha$ then follows as in the case of $C$ smooth and $D$ dualisable.

\end{proof}

Our main interest is in computing the Hochschild map $\HH{\C} \rightarrow \HH{\A}$ induced
by a continuous adjunction  $f: \C \leftrightarrow \A : f^{r}$ with smooth source and rigid target. By Corollary \ref{corep},
the induced $A$-linear functor $F=f_{\A}: C_{\A} \rightarrow \A$ has a left adjoint $F^{l}: \A \rightarrow \C_{\A}$ and 
$F$ is corepresentable by $F^{l}(\unit{\A})=E \in \C_{\A}$: $F \simeq \Homun{\A}{E}{-}$. Thus given a Hochschild class $k[i] \rightarrow \HH{\C}$ adjoint to a natural transformation $\Inv{\C}[i] \stackrel{{\alpha}}\rightarrow \Id{\C}$, we
get an induced natural transformation $\Inv{\C_{\A}/\A}[i] \stackrel{{\alpha}_{\A}}\rightarrow \Id{\C_{\A}}$ and hence an induced natural transformation $F\Inv{\C_{\A}}F^{r}[i] \stackrel{F {\alpha}_{\A} F^{r}}\rightarrow F F^{r}$. Post-composing with the counit $FF^{r} \rightarrow \Id{\A}$ and applying the tensor product $\mult{A} : \A \otimes \A \rightarrow \A$, we obtain a composition 
\[
F \Inv{\C_{\A}} F^{r}(\unit{\A})[i] \stackrel{{\alpha}_{\A}}\rightarrow F F^{r}(\unit{\A}) \rightarrow \mult{\A} \mult{A}^{r}(\unit{\A}).
\] 
Using the isomorphisms  $\Trr{\A}{F\Inv{C_{\A}} F^{r}} \simeq F F^{l}(\unit{\A}) \simeq  \Endun{\A}{E}$ and $\Trr{\A}{F F^{r}}\simeq F F^{r}(\unit{\A}) \simeq \Endun{\A}{E}^{\vee}$ from Corollary \ref{corep}, we obtain the composition 
\begin{equation}
\label{localhochmap}
\unit{\A}[i] \rightarrow \Endun{\A}{E}[i] \rightarrow \Endun{\A}{E}^{\vee} \rightarrow  \mult{\A} \mult{\A}^{r}({\unit{\A}}). 
\end{equation}
where $\unit{\A}[i] \rightarrow \Endun{\A}{E}[i]$ is the shifted unit map. Note that under the isomorphism $\Endun{\A}{E}^{\vee} \simeq \Hom{\A}{\Inv{\C_{\A}/\A}(E)}{E}$, the map $\Endun{\A}{E}[i] \rightarrow \Endun{\A}{E}^{\vee} \simeq \Homun{\A}{\Inv{\C_{\A}/\A}(E)}{E}$ identifies with $\Hom{\A}{-}{E}$ applied to $\alpha_{\A}: \Inv{\C_{\A}/\A}[i] \rightarrow \Id{\C_{\A}}$ evaluated on $E$.

\begin{prop}\label{hochrig}

Given a continuous adjunction $F: \C \leftrightarrow \A : F^{r}$ with smooth source and rigid
target, the image of a Hochschild chain adjoint to $\alpha : \Inv{\C}[i] \rightarrow \Id{\C}$ under
the induced map $\HH{\C} \rightarrow \HH{\A}$ is obtained by applying the functor
$\Hom{\A}{\unit{\A}}{-}: \A \rightarrow \Vect$ to the composition \eqref{localhochmap} and
precomposing with the unit $k \rightarrow \End{\A}{\unit{\A}}$. 
\end{prop}

\begin{proof}
Using the above isomorphisms and naturality of trace with respect to induction and restriction between $k$-linear 
and $A$-linear dg categories, we obtain a commutative diagram
\[
\xymatrix{
\Trr{k}{\Inv{\C}}[i] \ar[r] \ar[d] & \Trr{k}{\Id{\C}} \ar[d] \\
\res{\A}{k} \Trr{\A}{\Inv{\C_{\A}}}[i] \ar[r] \ar[d] & \res{\A}{k} \Trr{\A}{\Id{\C_{\A}}} \ar[d] \\
\res{\A}{k} \Trr{\A}{F\Inv{\C_{\A}}F^{r}}[i] \ar[r] \ar@{=}[d] &  \res{\A}{k} \Trr{\A}{F F^{r}} \ar@{=}[d] \\
\res{\A}{k} \Endun{\A}{E}[i] \ar[r] & \res{\A}{k} \Endun{\A}{E}^{\vee} \ar[r] & \res{\A}{k} \mult{\A} \mult{\A}^{r}({\unit{\A}}) \simeq \Trr{k}{\Id{\A}}
}
\]
Finally, note that the restriction functor $\res{\A}{k} : \A \rightarrow \Vect$ is just $\Hom{k}{\unit{\A}}{-}$
\end{proof}

\subsection{Functions on the loop space and Hochschild chains}

In order to encode the functoriality of base change maps \eqref{basechange}, it is best to use the $2$-category $\Corr$ of {\bf correspondences} with the symmetric monoidal structure induced by the Cartesian monoidal structure on affine schemes $\Aff$. At the level of homotopy categories, the objects of $\Corr$ are just affine schemes, a $1$-morphism in $\Corr$ from $U$ to $V$ is a correspondence
\[
\xymatrix{Z \ar[r]^{f} \ar[d]^{g} & U \\
V},
\]
and a $2$-morphism is a commutative diagram
\[\xymatrixrowsep{2.5pc}
\xymatrixcolsep{2.5pc}
\xymatrix{Z \ar[dr]^{h} \ar[drr]^{f} \ar[ddr]^{g} \\ & Z^{'} \ar[r]^{f^{'}} \ar[d]^{g^{'}}  & U \\
& V}
\]
with $h$ proper.

Composition of $1$-morphisms is given by pullback:
\[
\xymatrix{Z^{'} \times_{V} Z \ar[d] \ar[r]& Z \ar[r] \ar[d] & U \\
Z^{'} \ar[d] \ar[r] & V \\
W},
\]

It is easy to check that all objects $U \in \Corr$ are dualisable, with evaluation and coevaluation
\[
\xymatrix{U \ar[r]^-{\Delta} \ar[d] & U \times U & {} \\ {*} }
\xymatrix{U \ar[r] \ar[d]^{\Delta} & {*} \\ U \times U}
\]

Applying the formalism of traces from subsection \ref{subtraces}, we obtain that the trace of $\Id{U}$ in $\Corr$
is the correspondence
\[
\xymatrix{U \times_{U \times U} U \ar[r] \ar[d] & U \ar[d]^{\Delta} \ar[r] & {*} \\
U \ar[r]^-{\Delta} \ar[d] & U \times U \\
{*}}
\]
and is endowed with a natural $S^{1}$-action. Decomposing the circle $S^{1}$ into two intervals glued along their endpoints, one obtains an identification $\Mapun{S^{1}}{U} \simeq U \times_{U \times U} U \simeq \Trr{\Corr}{\Id{U}}$, and one can identify the natural $S^{1}$-action on $\Trr{\Corr}{\Id{U}}$ with `loop rotation' on $\Mapun{S^{1}}{U}$.

\begin{rem}
The formalism of correspondences makes sense for more general prestacks, usually with some restrictions on the arrows, but we shall only need to use it for affine schemes.
\end{rem}

As noted in \cite{gr1} 5.5.3, base change isomorphisms for ${\rm QCoh}$ give rise to a symmetric monoidal functor between $2$-categories

\begin{equation}
{\rm QCoh} : \Corr \rightarrow (\DGCattwo)^{\rm 2-op}
\end{equation}

Concretely, ${\rm QCoh} : \Corr \rightarrow (\DGCattwo)^{\rm 2-op}$ takes an object $U$ to $\QCoh{U}$, a morphism $V \stackrel{g}\leftarrow Z \stackrel{f} \rightarrow U$ to the functor $g_{*}f^{*} : \QCoh{U} \rightarrow \QCoh{V}$, and a $2$-morphism $h : Z \rightarrow Z^{'}$ to a natural transformation $g^{'}_{*}{f^{'}}^{*} \Rightarrow {g^{'}}_{*}h_{*}h^{*} {f^{'}}^{*} \simeq g_{*}f^{*}$ induced by the unit $\Id{Z^{'}} \Rightarrow h_{*}h^{*}$. 

\begin{rem}Note the contravariance between $h$ and the induced natural transformation. This is the reason for the `2-op' in $(\DGCattwo)^{\rm 2-op}$. Note that the `2-op' affects only the direction of functoriality of trace,
not the trace itself.
\end{rem}

We end this section with a comparison of geometrically and algebraically defined $S^{1}$-actions.

\begin{thm}\label{circacts}
For an affine scheme $U$, there is a natural isomorphism of $S^{1}$-complexes
\[
\Gamma(LU,\O_{LU}) \simeq \HH{\QCoh{U}}
\]
where the left-hand side has the $S^{1}$-action coming from the identification $LU=\Trr{\Corr}{\Id{U}}$ and the right-hand side has the $S^{1}$-action coming from the identification $\HH{\QCoh{U}}=\Trr{\DGCattwo}{\Id{\QCoh{U}}}$.
\end{thm}

\begin{proof}
Apply the naturality of $S^{1}$-actions from Proposition \ref{circleeq} to the symmetric monoidal 
functor ${\rm QCoh} : \Corr \rightarrow (\DGCattwo)^{\rm 2-op}$.
\end{proof}

\section{Shifted symplectic and Lagrangian structures on the moduli of objects}\label{sympstr}

\subsection{Graded $S^{1}$-complexes}

Given a group prestack $G$, recall that its classifying prestack is the geometric realisation of the corresponding simplicial prestack: $BG : = | \cdots  G \times G \triparr G \doubarr *|$. The dg-category of representations of $G$ is by definition the category of quasi-coherent sheaves on the classifying prestack $BG$: 
$\Rep{G} := \QCoh{BG}$.

Consider the quotient map $* \stackrel{q}\rightarrow BG$ and the map to a point $BG \stackrel{\pi}\rightarrow *$. We have adjoint pairs of functors
\begin{eqnarray*} q^{*} : \Rep{G}=\QCoh{BG} \leftrightarrow \Vect: q_{*} \\ \pi^{*} : \Vect \leftrightarrow \QCoh{BG}=\Rep{G} : \pi_{*}.
\end{eqnarray*}
In terms of representations, $q^{*}$ forgets the $G$-action, $q_{*}$ coinduces from the trivial group, $\pi^{*}$ gives the trivial representation, and $\pi_{*}$ takes $G$-invariants. For $G$ sufficiently nice, the right adjoints are continuous.

More generally, given a map between group prestacks $\varphi: G_{1} \rightarrow G_{2}$, we have an induced map $f :BG_{1} \rightarrow BG_{2}$ of classifying prestacks. In good circumstances, we have a continuous adjunction
$f^{*} : \Rep{G_{2}}=\QCoh{BG_{2}} \longleftrightarrow \QCoh{BG_{1}}: f_{*}$, which we refer to as {\bf restriction} and {\bf coinduction} of representations. \footnote{For classical group schemes, these functors correspond to the usual (derived) restriction and coinduction functors.}

In particular, consider the abelian group $S^{1}$ in $\Prstk$. We define an {\bf $S^{1}$-complex} to be a quasi-coherent sheaf on the classifying prestack $BS^{1}$. \footnote{It is easy to show that this category of  $S^{1}$-complexes is equivalent to others in the literature, for example, with the category of functors $\Fun{BS^{1}}{\Vect}$.} By \cite{bznloops} Corollary 3.11, applying $B$ to the affinisation map\footnote{Given a prestack $X$, the affinisation of $X$ is by definition the prestack $\Mapp{\calg}{\Gamma(X,\O_{X})}{-}: \calg^{\leq 0} \rightarrow \Spc$. It is not hard to show that the affinisation of $S^{1}$ is $B\Ga$. See \cite{bznloops}, Lemma 3.13.} $S^{1} \rightarrow B\Ga$ induces an equivalence under pullback
\[
\QCoh{B^{2}\Ga} \simeq \QCoh{BS^{1}}.
\]
We may therefore identify $S^{1}$-complexes with $B\Ga$-complexes, and we freely do so. We shall also be interested in {\bf graded $S^{1}$-complexes}, which by definition are objects of $\QCoh{B(B\Ga \rtimes \Gm)}$. \footnote{One can show that restriction of representations along $\Gm \rightarrow B\Ga \rtimes \Gm$ is conservative and preserves limits, so restriction/coinduction is comonadic in this case. Thus we may identify $\QCoh{B(B\Ga \rtimes \Gm)}$ with certain comodules in $\QCoh{B\Gm}$. One can use this to identify objects of $\QCoh{B(B\Ga \rtimes \Gm)}$ with $S^{1}$-complexes in $\QCoh{B\Gm}$, hence the name `graded $S^{1}$-complex'.}

Using the pullback square
\begin{equation}
\label{grpfibseq} 
\xymatrix{B^{2}\Ga \ar[r]^-{i} \ar[d]^{\pi} & B(B\Ga \rtimes \Gm) \ar[d]^{p} \\
{*} \ar[r]^{q} & B\Gm}
\end{equation}
and the section $j : B\Gm \rightarrow B(B\Ga \rtimes \Gm)$ of $p : B(B\Ga \rtimes \Gm) \rightarrow B\Gm$,
we can define various complexes and maps of complexes functorially associated to (graded) $S^1$-complexes.\footnote{Achtung: Quasi-coherent base change {\it does not} hold for the pullback square \ref{grpfibseq}.}
By definition, the {\bf negative cyclic complex} $\NC{E}$ of an $S^{1}$-complex $E \in \QCoh{B^{2}\Ga}$ is
the complex of $B\Ga$-invariants:
\[
\NC{E}:= \pi_{*}E \in \Vect.
\]
Similarly, given a graded $S^{1}$-complex $F \in \QCoh{B(B\Ga \rtimes \Gm)}$, we define its {\bf weight-graded negative cyclic complex} as the pushforward to $B\Gm$:
\[
\NCw{F}:= p_{*}F \in \Vect^{gr} \simeq \QCoh{B\Gm}.
\]

While the functors $q^* : \QCoh{B\Gm} \rightarrow \Vect$ and $i^*: \QCoh{B(B\Ga \rtimes \Gm)} \rightarrow \QCoh{B^{2}\Ga}$ are given concretely by {\it summing} over the weight-graded components of a graded (mixed) complex, for our purposes it will be more relevant to take the {\it product} over the weight-graded components. More formally, we note that the right adjoint functors $q_*: \Vect \rightarrow \QCoh{B\Gm}$ and $i_*: \QCoh{B^{2}\Ga} \rightarrow \QCoh{B(B\Ga \rtimes \Gm)}$ can be shown to be continuous and satisfy the projection formula (using \cite{dringaits}, Corollary 1.4.5, and the fact that the morphisms are qca), and hence themselves admit (non-continuous) right adjoints $(q_*)^r$ and $(i_*)^r$, which concretely are given by taking the product over weight-graded components. There are natural transformations 
\begin{align}
\begin{split}
& q^* \Rightarrow (q_*)^r \\
& i^* \Rightarrow (i_*)^r
\end{split}
\end{align}
concretely given by mapping the direct sum to the direct product. More precisely, the natural transformation $q^* \Rightarrow (q_*)^r$ is adjoint to a natural transformation $q_* q^* \Rightarrow \Id{\QCoh{B\Gm}}$
induced via the projection formula from the natural map $q_* q^* \O_{B\Gm} \rightarrow \O_{B\Gm}$ corresponding to the projection $k[t,t^{-1}] \rightarrow k$ of the regular representation onto the trivial representation. An analogous construction gives the natural transformation $i^* \Rightarrow (i_*)^r$.

The above long song and dance leads to the following simple and important observations.

\begin{lem}\label{pthcomp} 
Given a graded mixed complex $E \in \QCoh{B(B\Ga \rtimes \Gm)}$, there is a natural map
\[
\NC{i^{*}E} \rightarrow \prod_p \NCw{E}(p) 
\]
and so in particular a natural `pth component' map 
\begin{equation}\label{pthcomp}
\NC{i^{*}E} \rightarrow \NCw{E}(p)
\end{equation}
for each $p$.

Moreover, applying $p_{*}$  to the unit $\Id{\QCoh{B(B\Ga \rtimes \Gm)}} \Rightarrow j_* j^*$, we obtain a natural transformation $p_* \Rightarrow j^*$. Passing to weight-graded components, we obtain for each $p$ a natural map
\[
\NCw{E}(p) \rightarrow E(p).
\]

\end{lem}

\subsection{Closed differential forms}

Given an affine scheme $U$,  the map $S^{1} \rightarrow B\Ga$ induces an equivalence $\Map{B\Ga}{U} \simeq \Map{S^{1}}{U}=LU$, by definition of affinisation. The action of $B\Ga \rtimes \Gm$ on $B\Ga$ 
then induces an action of $B\Ga \rtimes \Gm$ on $LU$ and hence the functions on $LU$ carry a natural structure of graded $S^{1}$-module.\footnote{For a more detailed discussion in the not necessarily affine case, see Section 4 of \cite{bznloops}.} More formally, $LU$ is a $B\Ga \rtimes \Gm$-space, and we have a fibre square
\[
\xymatrix{LU \ar[d]^{r} \ar[r]^{p} & {*} \ar[d]^{q} \\
\widetilde{LU} \ar[r]^-{\pi} & B(B\Ga \rtimes \Gm)}
\]
where $\widetilde{LU}  \simeq LU/B\Ga \rtimes \Gm$. One can check that $q$ is a `good' morphism
\footnote{More precisely, $q$ is a `qca' morphism in the sense of \cite{dringaits}, since its fibre
is $B\Ga \rtimes \Gm$, which is qca.}, so that
base change in this fibre square gives an isomorphism
\[
\Gamma(LU,\O_{LU}) \simeq p_{*}\O_{LU} \simeq p_{*}r^{*}\O_{\widetilde{LU}} \simeq q^{*}\pi_{*} \O_{\widetilde{LU}}.
\]
We thus obtain a direct sum decomposition
\[
\Gamma(LU,\O_{LU}) = \bigoplus_p \Gamma(LU,\O_{LU})(p)
\]
into weight-graded components. On the other hand, we have isomorphisms
\begin{align*}
& \Gamma(LU,\O_{LU}) \simeq \Hom{\QCoh{LU}}{\O_{LU}}{\O_{LU}} \simeq \Hom{\IndCoh{LU}}{\omega_{LU}}{\omega_{LU}} \simeq \\
&  \Hom{U}{\pi_{*}\omega_{LU}}{\omega_{U}} \simeq \prod_{p} \Gamma(U,\Wedge^{p}\Cotang{U}[p]), 
\end{align*}
where the last isomorphism uses \ref{pbw} and base change along the diagonal $\Delta: U \rightarrow U \times U$. Altogether, we obtain an identification 
\[
 \Gamma(LU,\O_{LU})(p) \simeq \Gamma(U,\Wedge^{p}\Cotang{U}[p])
\]
of the weight-graded components of the functions on $LU$.\footnote{The fact that the direct sum and direct product agree depends on the fact that $\Cotang{U}$ is connective.}

We introduce the following terminology, following \cite{ptvv}:

The {\bf space of $p$-forms of degree $n$} on an affine scheme $U$ is 
\[
\dfor{p}{U}{n}:=|\Gamma(LU,\O_{LU})(p)[n-p]| \simeq |\Wedge^{p}\Cotang{U}[n]|
\]
The {\bf space of closed $p$-forms of degree $n$} on $U$ is
\[
\cldfor{p}{U}{n}:=|\NCw{\Gamma(LU,\O_{LU})}(p)[n-p]|
\]
The natural map $\NCw{\Gamma(LU,\O_{LU}}(p) \rightarrow \Gamma(LU,\O_{LU})(p)$ from the second part of Lemma \ref{pthcomp} induces a map
\[\cldfor{p}{U}{n} \rightarrow \dfor{p}{U}{n}
\]
giving the `underlying $p$-form' of a closed $p$-form. The constructions being functorial in $U$, we obtain
a map of prestacks
\begin{equation}\label{formmap}
\cldfor{p}{-}{n} \rightarrow \dfor{p}{-}{n}
\end{equation}
on $\Aff$.

Following \cite{ptvv}, for a general laft-def prestack $X$, we define the {\bf space of closed $p$-forms} and the {\bf space of $p$-forms}, as well as the map between them, by applying $\Map{X}{-}$ to \ref{formmap}:

\[\cldfor{p}{X}{n}=\Map{X}{\cldfor{p}{-}{n}} \rightarrow \Map{X}{\dfor{p}{-}{n}}= \dfor{p}{X}{n}.
\]

We now give the central construction of this paper.

For a prestack $X$, we tautologically write $X=\colim_{(\Aff/X)} U$. Then
\[
\cldfor{p}{X}{n-p} := \Map{X}{\cldfor{p}{-}{p-n}} \simeq \lim_{(\Aff/X)^{\rm op}}  \cldfor{p}{U}{p-n} \simeq 
\lim_{(\Aff/X)^{\rm op}}  |\NCw{U}(p)[-n]|.
\]
The universal continuous adjunction $\Funi_{\C} : \C \longleftrightarrow \Ind{\Perf{\M_{\C}}} : \Funi^{r}_{\C}$ gives
an $S^{1}$-equivariant map $\HH{\C} \rightarrow \HH{\Ind{\Perf{\M_{\C}}}}$. Composing with the natural $S^{1}$-equivariant map $\HH{\Ind{\Perf{\M_{\C}}}} \rightarrow \lim_{(\Aff/X)^{\rm op}} \HH{\QCoh{U}} \simeq
\lim_{(\Aff/X)^{\rm op}} \Gamma(LU,\O_{LU})$, taking invariants, and using Lemma \ref{pthcomp}, we obtain for each $p$ a natural map 
\[
\NC{\C} \rightarrow \lim_{(\Aff/X)^{\rm op}} \NCw{\Gamma(LU,\O_{LU})}(p).
\]
Truncating and shifting gives a map $\tilde{\kappa}_p: |\NC{\C}[-n]| \rightarrow \cldfor{p}{\M_{\C}}{p-n}$.
Similarly, define a map $\kappa_{p} : |\HH{\C}[-n]| \rightarrow \dfor{p}{\M_{\C}}{p-n}$.
Functoriality of invariants and of the $p$th component map \ref{pthcomp} gives the following.

\begin{prop}\label{negtoclos}
For each  $n \in \ZZ, p \in \NN$, there is a commutative square of spaces
\[
\xymatrix{|\NC{\C}[-n]| \ar[d] \ar[r]^{\tilde{\kappa}_p} & \cldfor{p}{\M_{\C}}{p-n} \ar[d] \\
|\HH{\C}[-n]| \ar[r]^{\kappa_p} & \dfor{p}{\M_{\C}}{p-n} }
\]
In words: from a negative cyclic class $\alpha: k[n] \rightarrow \NC{\C}$ of degree $n$, we obtain for each $p$ 
a closed $p$-form $\tilde{\kappa}(\alpha)_p$ of degree $p-n$ on the moduli space $\M_{\C}$, and the underlying $p$-form is associated to the underlying Hochschild class.
\end{prop}

We now describe how to compute the $p$-forms on $\M_{\C}$ corresponding to Hochschild classes $k[n] \rightarrow \HH{C}$, in the case of a smooth dg category $\C$. Using the isomorphism \ref{hochsmooth}, we  represent a Hochschild class by a map of endofunctors $\alpha: \Inv{\C}[n] \rightarrow \Id{\C}$. Inducing the universal continuous adjunction $\Funi_{\C} : \C \longleftrightarrow \Ind{\Perf{\M_{\C}}}: \Funi^{r}_{\C}$ along the symmetric monoidal functor $\Upsilon : \Ind{\Perf{\M_{\C}}} \rightarrow \IndCoh{\M_{\C}}$, we obtain a continuous adjunction
\[
{\Funny}_{\C} : \IndCoh{\M_{\C}} \otimes \C \longleftrightarrow \IndCoh{\M_{\C}}:{\Funny}^{r}_{\C}
\] 
in which the left adjoint $\Funny_{\C}$ is corepresentable by $\Upsilon(\Euni_{\C}) \in \IndCoh{\M_{\C}} \otimes \C$. Applying the induced map of endofunctors  $\tilde{\alpha}: \Inv{\C_{\M_{\C}}}[n] \rightarrow \Id{\C_{\M_{\C}}}$ to the object  ${\Funny}^{r}_{\C}(\omega_{\M_{\C}})$ followed by applying the functor $\widetilde{\Funi_{\C}}=\Homun{\M_{\C}}{\Upsilon\Euni_{\C}}{-}$, we obtain a map
\[
\Endun{\M_{\C}}{\Upsilon \Euni_{\C}}[n] \stackrel{\tilde{\alpha}}\rightarrow \Endun{\M_{\C}}{\Upsilon \Euni_{\C}}^{\vee}.
\]
Here we have used the isomorphisms $\Endun{\M_{\C}}{\Upsilon \Euni_{\C}} \simeq \Funny_{\C}\Inv{\C_{\M_{\C}}}\Funny^{r}_{\C}(\omega_{\M_{\C}})$ and $\Endun{\M_{\C}}{\Upsilon \Euni_{\C}}^{\vee} \simeq \Funny_{\C} \Funny^{r}_{\C}(\omega_{\M_{\C}})$ induced by \eqref{endisos}.
Pre-composing with the isomorphism \eqref{liemap} and the trace map of Corollary \ref{invserre}, we obtain a map
\begin{equation}\label{oneform}
\alpha_1 : \Tang{\M_{\C}}[-1+n]  \simeq \Endun{\M_{\C}}{\Upsilon \Euni_{\C}}[n] \stackrel{\tilde{\alpha}}\rightarrow {\Endun{\M_{\C}}{\Upsilon \Euni_{\C}}}^{\vee} \stackrel{\tr}\rightarrow \omega_{\M_{\C}}.
\end{equation}

\begin{prop}\label{pforms}
Let $\C$ be a smooth dg category. Given a Hochschild chain $k[n] \rightarrow \HH{\C}$ corresponding to a map of endofunctors $\alpha : \Inv{\C}[n] \rightarrow \Id{\C}$, the corresponding $1$-form of degree $1-n$ on $\M_{\C}$ is (dual to) the map $\alpha_{1}$ from \eqref{oneform}, while the corresponding $p$-form $\kappa_{p}(\alpha)$ of degree $p-n$ is (dual to) the composition
\[
\sym{p}{\Tang{\M_{\C}}[-1]}[n] \rightarrow {\Tang{\M_{\C}}[-1]}^{\otimes p}[n] \stackrel{\circ}\rightarrow \Tang{\M_{\C}}[-1][n] \stackrel{\tr}\rightarrow \omega_{\M_{\C}},
\]
where the map $\sym{p}{\Tang{\M_{\C}}[-1]} \rightarrow {\Tang{\M_{\C}}[-1]}^{\otimes p}$ is symmetrisation, the map ${\Tang{\M_{\C}}[-1]}^{\otimes p} \stackrel{\circ}\rightarrow \Tang{\M_{\C}}[-1]$ is the $p$-fold multiplication in the associative algebra structure on $\Tang{\M_{\C}}[-1]$, and the map $\Tang{\M_{\C}}[-1][n] \stackrel{\tr}\rightarrow \omega_{\M_{\C}}$
is induced by the trace map of Corollary \ref{invserre}.
\end{prop}

\begin{proof} The maps are defined globally, so to check that the composition is dual to that giving the $p$-form 
$\kappa_{p}(\alpha)$, it is enough to check this by restricting along each map $U \rightarrow \M_{\C}$ from an affine $U$ of finite type. For such a map, we use Lemma \ref{hochrig} on Hochschild maps with smooth source and rigid target. Taking the Grothendieck-Serre dual of this map as in Lemma \ref{deltadual} and using the isomorphism \ref{pbw} completes the identification of the $p$-form $\kappa_{p}(\alpha)$.
\end{proof}

\begin{rem}
In \cite{ptvv}, it is shown that if $X$ is locally an Artin stack, then $\dfor{p}{X}{n} \simeq |\Gamma(X,\Wedge^{p}\Cotang{X}[n])|$, so the above notion of the space of forms is at least reasonable in this case. Since the moduli space $\M_{C}$ is locally Artin when $C$ is of finite type, this will suffice for our purposes. For a general laft-def prestack, it is perhaps more natural to work directly with the Hodge filtration on de Rham cohomology. 
\end{rem}

\subsection{Symplectic and Lagrangian structures on the moduli of objects}

Recall from \cite{cy1} that a {\bf Calabi-Yau structure of dimension $d$} on a smooth dg category $\C$ is an $S^{1}$-equivariant map $\theta: k[d] \rightarrow \HH{\C}$ (equivalently, a map $k[d] \rightarrow \NC{\C}=\HH{\C}^{S^{1}}$) such that the corresponding map of endofunctors $\Inv{\C}[d] \rightarrow \Id{\C}$ is an isomorphism. More generally, given a continuous adjunction  $f: \C \leftrightarrow \D : f^{r}$ between smooth dg categories, a {\bf relative Calabi-Yau structure of dimension $d$} on the functor $f$ is a map $\eta: k[d] \rightarrow {\rm fib}(\NC{\C} \rightarrow \NC{\D})$ such that in the induced diagram
\begin{equation}\label{nondeg}
\xymatrix{
\Inv{\D}[d] \ar[r]\ar[d] & f \Inv{\C}[d] f^{r} \ar[r] \ar[d] & {\rm cof} \ar[d]\\
{\rm fib} \ar[r] & ff^{r} \ar[r] & \Id{\D}
}
\end{equation}
all vertical arrows are isomorphisms.\footnote{In \cite{cy1}, this was called a `left relative Calabi-Yau structure'. Since `right Calabi-Yau structures' do not appear explicitly in this paper, we drop `left'.} Here let us note that the map $\Inv{\D}[d] \rightarrow f \Inv{\C}[d] f^{r}$ 
is that given by \ref{shriekunit}.

In particular, a relative Calabi-Yau structure on $0 \rightarrow \D$ of dimension $d$ is just a Calabi-Yau structure of dimension $d+1$ on $\D$. We are especially interested in relative Calabi-Yau structures giving an absolute Calabi-Yau structure on $\C$. 

We have the following easy lemma, which will be used in the proof of the main theorem below.

\begin{lem}\label{leftimpliesright}
Let $\C$ and $\D$ be compactly generated smooth dg categories, $f : \C \rightarrow \D : f^{r}$ a continuous adjunction equipped with a relative Calabi-Yau structure of dimension $d$, and $x \in \D$ a right proper object so that $F_{\D}=\Hom{\D}{-}{x}^{*} : \D \rightarrow \Vect$ has continuous right adjoint $F^{r}_{\D}$. Then we have a commutative diagram 
\begin{equation*}
\xymatrix{
F_{\D}\Inv{\D}F^{r}_{D}[d] \ar[r]\ar@{}[d]|*=0[@]{\simeq}  & F_{\D}f \Inv{\C} f^{r}F^{r}_{\D}[d] \ar[r] \ar@{}[d]|*=0[@]{\simeq}  & F_{D}{\rm cof}F^{r}_{D} \ar@{}[d]|*=0[@]{\simeq} \\
F_{\D}{\rm fib}F^{r}_{\D} \ar[r] & F_{\D}ff^{r}F^{r}_{\D} \ar[r] & F_{\D}F^{r}_{\D}}
\end{equation*}
of endofunctors of $\Vect$ induced by applying $F^{r}_{\D}$ on the right and $F_{\D}$ the left of the diagram \ref{nondeg}. When evaluated on $k$, we obtain a commutative diagram
\begin{equation}\label{nondegfibres}
\xymatrix{
\End{\D}{x}[d] \ar[r]\ar@{}[d]|*=0[@]{\simeq}  & \End{\C}{f^{r}(x)}[d] \ar[r] \ar@{}[d]|*=0[@]{\simeq}  & \widetilde{\rm cof} \ar@{}[d]|*=0[@]{\simeq} \\
\widetilde{\rm fib} \ar[r] & \End{\C}{f^{r}(x)}^{*} \ar[r] & \End{\D}{x}^{*}}
\end{equation}
in which the upper left horizontal arrow is induced by applying the functor $f^{r}: \D \rightarrow \C$ and the lower right horizontal arrow is dual to that induced by $f^{r}$.
\end{lem}

\begin{proof}
If we define $F_{\C}:=F_{\D}f$, then for any compact object $y \in \C$, $F_{\C}(y)=\Hom{\D}{f(y)}{x}^{*} \simeq
\Hom{\C}{y}{f^{r}(x)}^{*}$, naturally in $y$, hence $F_{\C} \simeq \Hom{\C}{-}{f^{r}(x)}^{*}$. The other assertions then follow easily from Corollary \ref{corep}.
\end{proof}

We are now ready to prove the main theorem of this paper.

\begin{thm}\label{mainthm}
\begin{enumerate}
\item Given a smooth dg category $\C$ with Calabi-Yau structure $\theta: k[d] \rightarrow \NC{\C}$ of dimension $d$, 
the corresponding closed $2$-form  $\tilde{\kappa}_{2}(\theta) \in \cldfor{2}{\M_{\C}}{2-d}$ is non-degenerate. In words, a Calabi-Yau structure of dimension $d$ on a smooth dg category $\C$ induces on the moduli space of objects $\M_{\C}$ a symplectic form of degree $2-d$.

\item Given a continuous adjunction $f: \C \longleftrightarrow \D : f^{r}$ between smooth dg categories equipped with a relative Calabi-Yau structure  $\eta: k[d] \rightarrow {\rm fib}(\NC{\C} \rightarrow \NC{\D})$ of dimension $d$ that agrees with the absolute Calabi-Yau structure $\theta$, there is an induced Lagrangian structure on the map of moduli spaces
\[
\M_{\D} \rightarrow \M_{\C}.
\]
\end{enumerate}
\end{thm}

\begin{proof}
The proof of 1) is immediate from Lemma \ref{pforms}: the pairing $\Tang{\M_{\C}}[-1] \otimes \Tang{\M_{\C}}[-1] \rightarrow \omega_{\M_{\C}}[-d]$ given by the underlying $2$-form is exactly the Serre pairing of Corollary \ref{invserre}, after using the isomorphism $\Endun{\M_{\C}}{\Upsilon \Euni} \simeq \Tang{\M_{\C}}[-1]$. 

For the proof of 2), we have to describe the induced isotropic structure. For this, we use naturality of the map $|\NC{-}[d]| \rightarrow \cldfor{2}{\M_{-}}{2-d}$ to obtain a diagram of fibre sequences
\[\xymatrix{
{\rm fib}(|\NC{f}|) \ar[r]\ar[d] & |\NC{\C}[-d]| \ar[r]^{|\NC{f}|}\ar[d] & |\NC{\D}[-d]| \ar[d] \\
{\rm fib}({\varphi^{*}}_{\rm cl}) \ar[r]\ar[d] & \cldfor{2}{\M_{\C}}{2-d}  \ar[r]^{\varphi^{*}_{\rm cl}}\ar[d] & \cldfor{2}{\M_{\D}}{2-d} \ar[d] \\
{\rm fib}(\varphi^{*}) \ar[r] & \dfor{2}{\M_{\C}}{2-d}  \ar[r]^{\varphi^{*}} & \dfor{2}{\M_{\D}}{2-d} 
}
\]
The relative Calabi-Yau structure $k \rightarrow {\rm fib}(|\NC{\C}[-d]| \rightarrow |\NC{\D}[-d]|)$ determines
a point in ${\rm fib}(|\NC{f}|)$, which maps under the upper left vertical arrow to a point in ${\rm fib}({\varphi^{*}}_{\rm cl})$, determining an isotropic structure. 

To prove non-degeneracy of the isotropic structure, note that the maps of functors from Lemma \ref{tangmap}, together with the relative Calabi-Yau structure, induce a commutative diagram of functors
\[
\xymatrix{\Funi_{\D}\Inv{\D}\Funi^{r}_{\D}[d] \ar[r] \ar[d] & \varphi^{*}\Funi_{\C}\Inv{\C}\Funi^{r}_{\C}[d] \ar[r] \ar@{}[d]|*=0[@]{\simeq} & {\rm fib} \ar[d] \\
{\rm cof} \ar[r] & \varphi^{*}\Funi_{\C}\Funi^{r}_{\C} \ar[r] & \Funi_{\D}\Funi^{r}_{\D}}.
\]
Evaluating this diagram on $\O_{\M_{\D}}$ and applying $\Upsilon$ gives a commutative diagram
\[
\xymatrix{\Tang{\M_{\D}}[-1+d] \ar[r]\ar[d] & \phi^{!}\Tang{\M_{\C}}[-1+d] \ar[r]  \ar@{}[d]|*=0[@]{\simeq}  & \Tang{\M_{\D}/\M_{\C}}[-1+d] \ar[d] \\
\Upsilon \Cotang{\M_{\D}/\M_{\C}}[1] \ar[r] & \varphi^{!}\Upsilon \Cotang{\M_{\C}}[1] \ar[r] & \Upsilon \Cotang{\M_{\D}}[1]}
\]
in which the upper left horizontal arrow is the shifted tangent map and the lower right horizontal arrow is the shifted cotangent map. 

It remains to see that the outer two vertical arrows in the above diagram are isomorphisms. Since $\M_{\D}$ is laft, it is enough to check isomorphisms on fibres over $k$-points $x \in \M_{\D}$, which by definition of the moduli space correspond to right proper objects $x \in \D$ giving dg functors $F_{\D}=\Hom{\D}{-}{x}^{*}: \D \rightarrow \Vect$ with continuous right adjoint. By Lemma \ref{tangmap}, the fibre of the upper left horizontal arrow is the map $\End{\D}{x}[d] \rightarrow \End{\C}{f^{r}(x)}[d]$ induced by the functor $f^{r} : \D \rightarrow \C$ and the fibre of the lower right horizontal map is dual to that, up to a shift. That the fibres of the outer two vertical maps are isomorphisms now follows from Lemma \ref{leftimpliesright}.
\end{proof}

\section{Applications and examples}

In this section, we apply Theorem \ref{mainthm} to a number of examples of relative Calabi-Yau structures on functors $\C \rightarrow \D$ to produce Lagrangian structures on the corresponding maps of moduli spaces $\M_{\D} \rightarrow \M_{\C}$. The example of local systems on manifolds with boundary and some version of the example of ind-coherent sheaves on Gorenstein schemes with anti-canonical divisors are also treated by Calaque \cite{calaque}, using different methods. The example coming from $A_{n}$-quivers was known in some form to experts. See for example \cite{toenreview}, 5.3.

\subsection{Oriented manifolds and Calabi-Yau schemes}

Given a closed oriented manifold $M$ of dimension $d$, Cohen-Gantra \cite{cohenganatra} constructed an absolute Calabi-Yau structure on the dg category $\Loc{M}$ of local systems 
on $M$. More generally, given an oriented manifold $N$ of dimension $d+1$ with boundary $\partial N=M$, Theorem 5.7 of \cite{cy1} gives a relative Calabi-Yau structure of dimension $d+1$ on the induction functor
\begin{equation}
\label{induct}
i_{!} : \Loc{\partial N} \rightarrow \Loc{N}.
\end{equation}
Applying Theorem \ref{mainthm} to this relative Calabi-Yau structure, we obtain the following.

\begin{cor}\label{locsys}
The relative Calabi-Yau structure on the functor \ref{induct} induces a Lagrangian structure on the corresponding map of moduli spaces 
\[
\M_{\Loc{N}} \rightarrow \M_{\Loc{\partial N}}.
\]
\end{cor}

Similarly, given a finite type Gorenstein scheme $X$ of dimension $d$ together with a trivialisation $\theta : \O_{X} \simeq K_{X}$ of its canonical bundle, Proposition 5.12 of \cite{cy1} gives an absolute Calabi-Yau structure of dimension $d$ on $\IndCoh{X}$. Given a Gorenstein scheme $Y$ of dimension $d+1$ with an anticanonical section $s \in K_{Y}^{-1}$ having a zero-scheme $X$ of dimension $d$, there is an induced trivialisation $\theta : \O_{X} \simeq K_{X}$, and Theorem 5.13 of \cite{cy1} gives a relative Calabi-Yau structure of dimension $d+1$ on the pushforward functor
\[\label{push}
i_{*} : \IndCoh{X} \rightarrow \IndCoh{Y}.
\]
Applying Theorem \ref{mainthm} to this relative Calabi-Yau structure, we obtain the following.

\begin{cor}\label{antican}
The relative Calabi-Yau structure on the functor \ref{push} induces a Lagrangian structure on the corresponding map of moduli spaces 
\[
\M_{Y} \rightarrow \M_{X}
\]
\end{cor}

\subsection{Lagrangian correspondences and exact sequences}

One of the basic examples of a relative Calabi-Yau structure, treated in \cite{cy1}, Theorem 5.14, comes from the representation theory of quivers of type $A_{n}$. Specifically, there is a natural functor 
\[\label{quivrelcy}
\amalg_{i=1}^{n+1} \Vect \rightarrow \Mod{A_{n}}
\]
with a relative Calabi-Yau structure of dimension $1$. Denoting the moduli space of objects in $\Vect$ by $\M_{1}$ and the moduli space of objects in $\Mod{A_{n}}$ by $\M_{n}$, Theorem \ref{mainthm} endows the induced map
\[
\M_{n} \rightarrow \Pi_{i=1}^{n+1} \M_{1}
\]
with a Lagrangian structure.

Let us explain the case $n=2$ in more detail. For the quiver $A_{2}$, we have two simple modules $S_{1}$ and $S_{2}$, which we denote schematically by $k \rightarrow 0$ and $0 \rightarrow k$ respectively, and the extension $P$ of $S_{1}$ by $S_{2}$, denoted schematically as $k \rightarrow k$. 

The functor
\[
\amalg_{i=1}^{3} \Vect \rightarrow \Mod{A_{2}}.
\]
taking the first copy of $k$ to the simple module $S_{1}$, the second copy of $k$ to $P$, and the third copy of $k$ to the simple module $S_{2}$ carries an essentially unique relative Calabi-Yau structure. Indeed, there is an isomorphism of $S^{1}$-complexes $\HH{\amalg_{i=1}^{3} \Vect} \simeq k \oplus k \oplus k$ given by the classes of the three copies of $k$, and similarly an isomorphism $\HH{\Mod{A_{2}}} \simeq k \oplus k$
given by the classes of $S_{1}$ and $S_{2}$. With respect to these isomorphisms, the exact sequence
$\HH{\Mod{A_{2}},\amalg_{i=1}^{3} \Vect}[-1] \rightarrow \HH{\amalg_{i=1}^{3} \Vect} \rightarrow \HH{\Mod{A_{2}}}$ identifies with the exact sequence
\[
\xymatrix{
k \ar[r] \ar[r]^{\spmat{1 \\ -1 \\ 1}}  & k \oplus k \oplus k \ar[r]^{\spmat{1 & 1 & 0 \\ 0 & 1 & 1}} & k \oplus k
}.
\]
By examining the action of the relevant functors on the simple modules of $A_{2}$, it is not hard to check that the identification $k \simeq \HH{\Mod{A_{2}},\amalg_{i=1}^{3} \Vect}[-1]$ satisfies the non-degeneracy necessary for a relative Calabi-Yau structure. 

Now consider the induced map $\M_{2} \rightarrow \M_{1} \times \M_{1} \times \M_{1}$. A $k$-point in $\M_{2}$ is a continuous functor $\Mod{A_{2}} \rightarrow \Vect$ with continuous right adjoint. The image of the exact sequence $S_{2} \rightarrow P \rightarrow S_{1}$ under this functor essentially determines the functor, and so we can consider $\M_{2}$ as the moduli space of exact sequence, with the first and last factor of $\M_{2} \rightarrow \M_{1} \times \M_{1} \times \M_{1}$ picking out the beginning and end of the sequence and the middle factor giving the middle term of the sequence. 

Note that the Lagrangian structure on the map $\M_{2} \rightarrow \M_{1} \times \M_{1} \times \M_{1}$ is with respect to the degree $2$ symplectic form $(\omega,-\omega,\omega)$ on the target, where $\omega$ is the standard degree $2$ symplectic form on $\M_{1}$.

We consider now a generalisation of the above construction to the moduli space of $A_{n}$-representations in a Calabi-Yau category $\C$ of dimension $d$. 

\begin{lem}\label{kunneth}
Given dualisable dg categories $\C$ and $\D$, there is a K\"unneth isomorphism $\HH{\C \otimes \D} \simeq \HH{\C} \otimes \HH{\D}$ of $S^{1}$-complexes. When $\C$ and $\D$ are smooth, the underlying $k$-linear K\"unneth isomorphism factors as $\HH{\C \otimes \D} \simeq \Hom{\endo{\C \otimes \D}}{\Inv{\C \otimes \D}}{\Id{\C \otimes \D}} \simeq \Hom{\endo{\C}}{\Inv{\C}}{\Id{\C}} \otimes \Hom{\endo{\D}}{\Inv{\D}}{\Id{\D}} \simeq \HH{\C} \otimes \HH{\D}$.
\end{lem}

\begin{proof}
The general K\"unneth theorem for traces follows from the trace formalism that we reviewed in Section \ref{subtraces}. The underlying $k$-linear isomorphism comes from the identification $(\C \otimes \D)^{\vee} \otimes (\C \otimes \D) \simeq (\D^{\vee} \otimes \D) \otimes (\C^{\vee} \otimes \C)$ and the corresponding identification $\ev{\C \otimes \D} \simeq \ev{\C} \otimes \ev{\D}$. In the case of smooth categories, passing to left adjoints gives a corresponding identification $\Inv{\C \otimes \D} \simeq \Inv{\C} \otimes \Inv{\D}$, whence the 
second claim follows. 
\end{proof}

\begin{prop}\label{tensorrelcy}
Given smooth dg categories $A$, $B$, and $C$ with a relative Calabi-Yau structure $\theta_1 \in \Hom{S^{1}}{k[d_1]}{\HH{B,A}}$ of dimension $d_1$ on a functor $f : A \rightarrow B$, and an absolute Calabi-Yau structure $\theta_2 \in \Hom{S^{1}}{k[d_2]}{\HH{\C}}$ of dimension $d_2$, the tensor product $f \otimes \Id{\C} : A \otimes C \rightarrow B \otimes \C$ has an induced relative Calabi-Yau structure $\theta_1 \otimes \theta_2$ of dimension $d_1 + d_2$. 

In particular, setting $A=0$, we see that the tensor product of two dg categories with Calabi-Yau structures has an induced Calabi-Yau structure.
\end{prop}

\begin{proof}
This follows easily from the K\"unneth formula of Lemma \ref{kunneth}.
\end{proof}

We state explicitly an important special case of Proposition \ref{tensorrelcy}.

\begin{cor}\label{stackexact}
Let $(\C,\theta)$ be a non-commutative Calabi-Yau of dimension $d$ and set $\C_{n}=\Mod{A_{n}} \otimes \C$. Then the functor
\[
\amalg_{i=1}^{n+1} \C \rightarrow \C_{n}
\]
induced by tensoring \ref{quivrelcy} with $(\C,\theta)$ carries a relative Calabi-Yau structure of dimension $d+1$, 
and the induced map of moduli
\[
\M_{\C_{n}} \rightarrow \prod_{i=1}^{n+1} \M_{\C}
\]
carries a Lagrangian structure with respect to the degree $2-d$ symplectic structure on $\M_{\C}$.\footnote{Note: The exact form of the relative Calabi-Yau structure on \ref{quivrelcy} introduces a sign into one of the factors of 
the symplectic structure on $\prod_{i=1}^{n+1} \M_{\C}$}.
\end{cor}

 \bibliographystyle{plain}
\bibliography{cy2bib} 

\begin{thebibliography}{10}

\bibitem{amr}
David Ayala, Aaron Mazel-Gee, and Nick Rozenblyum.
\newblock Factorization homology of enriched?-categories.
\newblock {\url{arXiv:1710.06414}}.

\bibitem{bzntraces}
David Ben-Zvi and David Nadler.
\newblock Nonlinear traces.
\newblock {\url{arXiv:1305.7175}}.

\bibitem{bznloops}
David Ben-Zvi and David Nadler.
\newblock {L}oop spaces and connections.
\newblock {\em Journal of Topology}, 5(2):377--430, 2012.

\bibitem{cy1}
C.~Brav and T.~Dyckerhoff.
\newblock Relative {C}alabi-{Y}au structures.
\newblock {\url{arXiv:1606.00619 }}.

\bibitem{calaque}
Damien Calaque.
\newblock {L}agrangian structures on mapping stacks and semi-classical {TFT}s.
\newblock {\em Stacks and categories in geometry, topology, and algebra},
  643:1--23, 2015.

\bibitem{cohenganatra}
Ralph Cohen and Sheel Ganatra.
\newblock {C}alabi-{Y}au categories, the {F}loer theory of a cotangent bundle,
  and the string topology of the base.
\newblock {\url{http://math.stanford.edu/~ralph/scy-floer-string_draft.pdf}}.

\bibitem{dringaits}
Vladimir Drinfeld and Dennis Gaitsgory.
\newblock On some finiteness questions for algebraic stacks.
\newblock {\em Geometric and Functional Analysis}, 23(1):149--294, 2013.

\bibitem{gr1}
Dennis Gaitsgory and Nick Rozenblyum.
\newblock {\em A {S}tudy in {D}erived {A}lgebraic {G}eometry: {V}olume I:
  {C}orrespondences and {D}uality}, volume 221 of {\em Mathematical Surveys and
  Monographs}.
\newblock 2017.

\bibitem{gr2}
Dennis Gaitsgory and Nick Rozenblyum.
\newblock {\em A Study in {D}erived {A}lgebraic {G}eometry. {V}olume II:
  {D}eformations, {L}ie Theory and {F}ormal {G}eometry}, volume 221 of {\em
  Mathematical Surveys and Monographs}.
\newblock 2017.

\bibitem{goldman}
William~M Goldman.
\newblock The symplectic nature of fundamental groups of surfaces.
\newblock {\em Advances in Mathematics}, 54(2):200--225, 1984.

\bibitem{henn}
Benjamin Hennion.
\newblock {T}angent {L}ie algebra of derived {A}rtin stacks.
\newblock {\em Journal f{\"u}r die reine und angewandte Mathematik (Crelles
  Journal)}, 2018(741):1--45, 2018.

\bibitem{hss}
Marc Hoyois, Sarah Scherotzke, and Nicol{\`o} Sibilla.
\newblock Higher traces, noncommutative motives, and the categorified {C}hern
  character.
\newblock {\em Advances in Mathematics}, 309:97--154, 2017.

\bibitem{sjf}
Theo Johnson-Freyd and Claudia Scheimbauer.
\newblock (op)lax natural transformations, twisted quantum field theories, and
  ``even higher'' {M}orita categories.
\newblock {\em Advances in Mathematics}, 307:147--223, 2017.

\bibitem{kondprih}
Grigory Kondyrev and Artem Prihodko.
\newblock Categorical proof of holomorphic {A}tiyah-{B}ott formula.
\newblock {\url{arXiv:1607.06345}}.

\bibitem{loday}
Jean-Louis Loday.
\newblock {\em Cyclic homology}, volume 301.
\newblock Springer Science \& Business Media, 2013.

\bibitem{lurHA}
Jacob Lurie.
\newblock Higher algebra.
\newblock {\url{http://www.math.harvard.edu/~lurie/papers/HA.pdf}}.

\bibitem{lurHT}
Jacob Lurie.
\newblock {\em Higher Topos Theory (AM-170)}.
\newblock Princeton University Press, 2009.

\bibitem{mukai}
Shigeru Mukai.
\newblock Symplectic structure of the moduli space of sheaves on an abelian or
  {K}3 surface.
\newblock {\em Inventiones mathematicae}, 77(1):101--116, 1984.

\bibitem{ptvv}
Tony Pantev, Bertrand To{\"e}n, Michel Vaqui{\'e}, and Gabriele Vezzosi.
\newblock Shifted symplectic structures.
\newblock {\em Publications math{\'e}matiques de l'IH{\'E}S}, 117(1):271--328,
  2013.

\bibitem{shendetakeda}
Vivek Shende and Alex Takeda.
\newblock Symplectic structures from topological {F}ukaya categories.
\newblock {\url{arXiv:1605.02721}}.

\bibitem{toenreview}
Bertrand To{\"e}n.
\newblock Derived algebraic geometry.
\newblock {\em EMS Surveys in Mathematical Sciences}, 1(2):153--240, 2014.

\bibitem{tv}
Bertrand To{\"e}n and Michel Vaqui{\'e}.
\newblock Moduli of objects in dg-categories.
\newblock {\em Annales scientifiques de l' \'Ecole normale sup{\'e}rieure},
  40(3):387--444, 2007.

\bibitem{tvchern}
Bertrand To{\"e}n and Gabriele Vezzosi.
\newblock Caract\`eres de {C}hern, traces {\'e}quivariantes et
  g{\'e}om{\'e}trie alg{\'e}brique d{\'e}riv{\'e}e.
\newblock {\em Selecta Mathematica}, 21(2):449--554, 2015.

\bibitem{waikit}
Wai-kit Yeung.
\newblock Weak {C}alabi-{Y}au structures and moduli of representations.
\newblock {\url{arXiv:1802.05398}}.

\end{thebibliography}

\end{document}